\documentclass{amsart}
\usepackage[latin1]{inputenc}
\usepackage[english]{babel}
\usepackage[T1]{fontenc}
\usepackage{amsmath}
\usepackage{amssymb}
\usepackage{amsthm}
\usepackage{array,booktabs}
\usepackage{graphicx}
\usepackage{lmodern}
\usepackage{tcolorbox}
\usepackage{quoting}
\usepackage{enumerate}
\usepackage{mathtools}
\usepackage{csquotes}
\usepackage{cite}
\usepackage{tikz}
\usepackage{tikz-cd}
\usetikzlibrary{matrix,arrows,decorations.pathmorphing}
\usepackage{wrapfig}

\renewcommand{\Im}{\operatorname{Im}}
\def\R{\ensuremath\mathbb{R}}
\def\C{\ensuremath\mathbb{C}}
\def\A{\ensuremath\mathbb{A}}
\def\Z{\ensuremath\mathbb{Z}}
\def\Q{\ensuremath\mathbb{Q}}
\def\N{\ensuremath\mathbb{N}}
\def\Hb{\ensuremath\mathbb{H}}

\usepackage{hyperref}
\usepackage[final]{pdfpages}
\usepackage{verbatim}
\newtheorem{thm}{Theorem}[section]

\newtheorem{cor}[thm]{Corollary}
\newtheorem{lemma}[thm]{Lemma}
\newtheorem{prop}[thm]{Proposition}
\theoremstyle{remark}
\newtheorem{remark}[thm]{Remark}

\def\eps{\ensuremath\varepsilon}
\def\0{\emptyset}

\def\Cl{\text{\rm Cl}}
\def\fin{\text{\rm fin}}
\def\sgn{\mathrm{sign}}

\def\Gal{\text{\rm Gal}}
\def\sym{\mathrm{sym}}
\def\SO{\hbox{\rm SO}}
\def\OO{\hbox{\rm O}}
\def\SL{\hbox{\rm SL}}
\def\PSL{\mathrm{PSL}}
\def\GL{\mathrm{GL}}
\def\modulo{\text{ \rm mod }}
\def\vol{\text{\rm vol}}

\def\pmod{\text{ \rm mod }}
\numberwithin{equation}{section}

\numberwithin{equation}{section}
\begin{document}
\title[Wide moments]{Wide moments of $L$-functions I: Twists by class group characters of imaginary quadratic fields}
\author{Asbj\o rn Christian Nordentoft}

\address{Mathematical Institute of the University of Bonn, Endenicher Allee 60, Bonn 53115, Germany }
%IDEAS FOR IMPROVEMENTS:
%-Write main moment calculation in a "informal" version
%-Make comment with the exponents upon assuming Lindelöf and sup norm
\email{\href{mailto:acnordentoft@outlook.com}{acnordentoft@outlook.com}}

\date{\today}
\subjclass[2010]{11F67(primary), and 11M41(secondary)}
\maketitle
\begin{abstract}
We calculate certain {\lq\lq}wide moments{\rq\rq} of central values of Rankin--Selberg $L$-functions $L(\pi\otimes \Omega, 1/2)$ where $\pi$ is a cuspidal automorphic representation of $\GL_2$ over $\Q$ and $\Omega$ is a Hecke character (of conductor $1$) of an imaginary quadratic field. This moment calculation is applied to obtain {\lq\lq}weak simultaneous{\rq\rq} non-vanishing results, which are non-vanishing results for different Rankin--Selberg $L$-functions where the product of the twists is trivial. %Here {\lq\lq}weak simultaneous{\rq\rq} means that non-vanishing results for different Rankin--Selberg $L$-functions are obtained where the product of the twists is trivial. 

The proof relies on relating the wide moments of $L$-functions to the usual moments of automorphic forms evaluated at Heegner points using Waldspurger's formula. To achieve this, a classical version of Waldspurger's formula for general weight automorphic forms is derived, which might be of independent interest. A key input is equidistribution of Heegner points (with explicit error-terms) together with non-vanishing results for certain period integrals. In particular, we develop a soft technique for obtaining non-vanishing of triple convolution $L$-functions.     \end{abstract}
\section{Introduction}
Determining the moments of central values of families of automorphic $L$-functions has a long history starting with the work of Hardy and Littlewood on the Riemann zeta function: 
$$ \int_0^T |\zeta(1/2+it)|^2 dt\sim T\log T,  $$
as $T\rightarrow \infty$ (see \cite[Chapter VII]{Titchmarsh86}). By now there exists precise conjectures for all moments of families of $L$-functions \cite{CFKRS05} with fascinating connections to Random Matrix Theory \cite{KeaSna00}. These moment conjectures are of deep arithmetic importance through their connections to the important topics of non-vanishing and subconvexity (see e.g. \cite{BlFoKoMiMiSa18}), which in turn are connected to respectively, rational points on elliptic curves (via the B--S-D conjectures, see \cite{Kolyvagin88}) and equidistribution problems (via the Waldspurger formula, see \cite{MichVenk06}). %The conjectures are however notoriously difficult 

In this paper we will calculate what we call \emph{wide moments} of central values of Rankin--Selberg $L$-functions $L(\pi \otimes \Omega,1/2)$, where $\pi$ is a cuspidal automorphic representation of $\GL_2$ with trivial central character of even lowest weight $k_\pi$ and $\Omega$ is a Hecke character of an imaginary quadratic field $K$ with infinity type $\alpha \mapsto (\alpha/|\alpha|)^k$ for some even integer $k\geq k_\pi$. More precisely, we will study the {\lq\lq}canonical{\rq\rq} square-roots of the central values via their connections to Heegner periods as in the work of Waldspurger \cite{Waldspurger85}. We will use these moment calculations to obtain a number of new non-vanishing results (see Subsection \ref{sec:applnonv} for the statements) of a certain kind that we call \emph{weak simultaneous non-vanishing}. In view of the Bloch--Kato conjectures, these non-vanishing results imply (in the holomorphic case) vanishing for certain twisted Selmer groups as explained in \cite[Corollary B]{BurungaleHida16}. 
\subsection{Wide moments of $L$-functions}
This paper is the first in a series of papers concerned with obtaining asymptotic evaluations of \emph{wide moments} of automorphic $L$-function. In all of the cases we will consider, these wide moments are connected to the usual moments of certain underlying periods of automorphic forms (in the case of this paper through the Waldspurger formula), which are much better behaved than the $L$-functions themselves. In particular, one can use a variety of more geometrically flavoured methods to study the distributional properties of these periods. 
%This will then allow us to calculate the wide moments. 

The abstract set-up is as follows: Given a finite abelian group $G$ with (unitary) dual $\widehat{G}$, we define 
\begin{equation} \label{eq:Wide}  \mathbf{Wide}(\widehat{G},n):=\{(\chi_1,\ldots,\chi_n)\in (\widehat{G})^n: \chi_1\cdots \chi_n=1\}.\end{equation}
Given maps $L_1,\ldots, L_n: G\rightarrow \C$ with Fourier transforms 
$$\widehat{L_i}:\widehat{G}\rightarrow \C,\quad  \chi\mapsto \frac{1}{|G|}\sum_{g\in G} L_i(g)\overline{\chi}(g),\quad i=1,\ldots, m,$$ 
we define the {\bf wide moment} of $\widehat{L_1},\ldots, \widehat{L_n}$, as the following expression:
\begin{equation}\label{eq:widemoment} \sum_{(\chi_i)_{1\leq i \leq n}\in \mathbf{Wide}(\widehat{G},n)}\,\prod_{i=1}^n \widehat{L_i}(\chi_i).  \end{equation}
Note that for $n=2$ and $\widehat{L_1}=\widehat{L_2}$ equivariant with respect to inverses (i.e. $\widehat{L_1}(\chi^{-1})=\overline{\widehat{L_1}(\chi)}$), we recover the usual second moment. The key point is that (\ref{eq:widemoment}) is equal to
\begin{equation}\label{eq:widemoment2} \frac{1}{|G|}\sum_{g\in G}\, \prod_{i=1}^n L_i(g),\end{equation}
(for $n=2$ this is exactly Plancherel). A nice way to see that (\ref{eq:widemoment})$=$(\ref{eq:widemoment2}) is to use that the Fourier transform takes products to convolutions, and (\ref{eq:widemoment}) is exactly the $n$-fold convolution product of $\widehat{L_1},\ldots, \widehat{L_n}$ evaluated at $\chi=1$. In the setting of automorphic $L$-functions one can in many cases calculate the wide moments (\ref{eq:widemoment}) using that the dual moments (\ref{eq:widemoment2}) are much better behaved. % Note that the expression (\ref{eq:widemoment}) is a usual $(1,\ldots, 1)$-moment of the Fourier transforms.

The first example in the literature of an asymptotic evaluation of a (higher) wide moment of automorphic $L$-functions, seems to be the work of Bettin \cite{Be17} on Dirichlet $L$-functions (note that here the terminology {\lq\lq}iterated moments{\rq\rq} is used):
\begin{align} 
\nonumber&\frac{1}{(p-2)^{n-1}}\sideset{}{^*}\sum_{(\chi_i)\in \mathbf{Wide}(p,n)} |L(\chi_1,1/2)|^2\cdots  |L(\chi_n,1/2)|^2\\
\label{eq:Bettin}&= c_{n,n}(\log p)^n+c_{n,n-1}(\log p)^{n-1}+\ldots +c_{n,0}+O(p^{-\delta}),\end{align}
as $p\rightarrow \infty$ with $p$ prime, for some $\delta >0$ and $c_{n,i}\in \R$. Here the asteriks on the sum means that the summation is restricted to primitive Dirichlet characters and we used the shorthand $\mathbf{Wide}(p,n):=\mathbf{Wide}(\widehat{(\Z/p\Z)^\times)},n)$. This result is a corollary of the moment calculation of the \emph{Estermann function} (which we think of as the underlying automorphic periods in this case). Another related result is the calculation of Chinta \cite{Chinta05} corresponding to a wide moment with $n=3$ for quadratic Dirichlet $L$-functions. 

The asymptotics (\ref{eq:Bettin}) was later generalized (with an extra average over the modulus $q$) by the author \cite[Corollary 1.9]{No19} to the wide moments of 
$$\widehat{L_1}(\chi)=\cdots=\widehat{L_n}(\chi)=L(f\otimes \chi,1/2),\quad \text{for }\chi:(\Z/q\Z)^\times\rightarrow \C^\times, $$
with $f$ a fixed holomorphic newform of even weight. In \cite{No19} the underlying automorphic periods are the \emph{additive twists} of $f$ (which reduces to \emph{modular symbols} for $k=2$). In forthcoming work the author joint with Drappeau calculate all moments of additive twists of level $1$ Maa{\ss} forms. 

The methods used to calculate the wide moments mentioned above are respectively, a classical approximate functional equation approach \cite{Be17}, multiple Dirichlet series \cite{Chinta05}, spectral theory \cite{Nordentoft20.4} (see also \cite{PeRi}), and dynamical systems (in the forthcoming work joint with Drappeau building on \cite{BeDr19}).   

\subsection{Main idea} Let us describe the main moment calculation of this paper in the simplest possible set-up. Let $ f:\Hb \rightarrow \C $ be a classical Hecke--Maa{\ss} eigenform of weight $0$ and (for simplicity) level $1$ (i.e. a real-analytic  joint eigenfunction for the hyperbolic Laplace operator and the Hecke operators which is invariant under $\PSL_2(\Z)$). Let $K$ be an imaginary quadratic field of discriminant $D_K<-6$ with class group $\Cl_K$. Given a class group character $\chi\in\widehat{\Cl_K}$, we denote by $L(f\otimes \chi, s)$ (the finite part of) the Rankin--Selberg $L$-function $L(f\otimes \theta_\chi, s)$, where $\theta_\chi$ is the theta series associated to $\chi$ of weight $1$ and level $|D_K|$ (equivalently we have $L(f\otimes \chi, s)=L(\pi_K\otimes \pi_\chi, s)$ where $\pi_K$ denotes the base change to $\GL_2(\A_K)$ of the automorphic representation corresponding to $f$ and $\pi_\chi$ is the automorphic representation of $\GL_1(\A_K)$ corresponding to $\chi$). A deep formula of Zhang \cite{Zh01}, \cite{Zh04} gives the following relation;
\begin{equation}\label{eq:Waldspurger1} \left|\sum_{[\mathfrak{a}]\in \Cl_K} f(z_{[\mathfrak{a}]})\chi([\mathfrak{a}])\right|^2=|c_f|^2 |D_K|^{1/2} L(f\otimes \chi,1/2), \end{equation}
where $\chi\in \widehat{\Cl_K}$ is a class group character of $K$, $z_{[\mathfrak{a}]}\in \PSL_2(\Z)\backslash \Hb$ denotes the Heegner point associated to $[\mathfrak{a}]\in \Cl_K$ and $c_f>0$ is a constant depending on $f$ (but independent of $\chi$). Using this relation together with orthogonality of characters and equidistribution of Heegner points, Michel and Venkatesh \cite{MichelVenk07} calculated the first moment of $L(f\otimes \chi,1/2)$, which they combined with subconvexity to obtain quantitive non-vanishing for these central values. This idea has since been generalized in many directions to obtain a variety of non-vanishing results \cite{DiPrSe15}, \cite{BurungaleHida16}, \cite{Khayutin20}, \cite{Templier11}. %It is also worth mentioning the work of Templier  where a similar result for the central derivative is obtain using a more classical approximate functional equation approach. 

We observe that (\ref{eq:Waldspurger1}) is exactly saying that the Fourier transform of 
$$\Cl_K\ni [\mathfrak{a}]\mapsto  |\Cl_K|f(z_{[\mathfrak{a}]}),$$ 
is given by a map of the form 
$$\widehat{\Cl_K}\ni \chi\mapsto \eps_{f, \chi}c_f |D_K|^{1/4} |L(f\otimes \overline{\chi},1/2)|^{1/2},$$ 
for some $\eps_{f, \chi}$ of norm $1$. Thus by the Fourier equality (\ref{eq:widemoment})$=$(\ref{eq:widemoment2}) and equidistribution of Heegner points due to Duke \cite{Du88} we conclude that for level $1$ Hecke--Maa{\ss} eigenforms $f_1,\ldots, f_n$ we have
\begin{align}\nonumber
\frac{|D_K|^{n/4}}{|\Cl_K|^n}\sum_{(\chi_i)\in \mathbf{Wide}(K,n)} \prod_{i=1}^n \eps_{f_i, \chi_i} c_{f_i}|L(f_i\otimes \chi_i,1/2)|^{1/2}&= \frac{1}{|\Cl_K|}\sum_{[\mathfrak{a}]\in \Cl_K}\prod_{i=1}^n f_i(z_{[\mathfrak{a}]})\\
 \label{eq:widemomentlevel1}&= \langle \prod_{i=1}^n f_i, \frac{3}{\pi}\rangle+o(1)  \end{align} 
as $|D_K|\rightarrow \infty$, where we use the short-hand $\mathbf{Wide}(K,n):=\mathbf{Wide}(\widehat{\Cl_K},n)$. This shows immediately that if $ \langle \prod_{i=1}^n f_i,1\rangle\neq 0$ then there exists
$$(\chi_1,\ldots, \chi_n)\in \mathbf{Wide}(K,n) \text{ s.t. }\prod_{i=1}^n L(f_i\otimes \chi_i,1/2)\neq 0.$$
We call the above {\it weak simultaneous non-vanishing} (see Section \ref{sec:weaknonv} for some background on this type of non-vanishing). 

\subsection{Non-vanishing results}\label{sec:applnonv}
The above proof sketch already gives new results. We will however push these ideas further in several aspects. First of all we deal with general weight forms (holomorphic or Maa{\ss}), which requires us to develop explicit Waldspurger type formulas in these cases (see Section \ref{sec:waldspurger}), which might be of independent interest. In particular, this requires studying Hecke characters which ramify at $\infty$, which leads to some complications. Secondly, we will obtain an explicit error-term in (\ref{eq:widemomentlevel1}), which requires bounding certain inner-products involving powers of the Laplace operator (see Section \ref{sec:technical}). This allows us to obtain non-vanishing results with uniformity in the spectral aspect. In particular in the case of width $n=2$ we obtain the following improved version of \cite[Theorem 1]{MichVenk06} with uniformity in the spectral parameter and allowing general weights. 

\begin{cor} \label{cor:n=2simpel}
Let $f$ be either a Hecke--Maa{\ss} cusp form of spectral parameter $t_f$ and level $1$ or a cuspidal holomorphic Hecke eigenform of weight $k_f$ and level $1$. Let $k$ be a positive even integer such that $k\geq k_f$ when $f$ is holomorphic. Put $T=|t_f|+k+1$ in the Maa{\ss} case and $T=k+1$ in the holomorphic case. 

Then for $\eps>0$ there exists a constant $c=c(\eps)>0$ such that for any imaginary quadratic field $K$ with discriminant $|D_K|\geq c T^{22+\eps}$, we have
$$ \#\{\chi\in \widehat{\Cl_K}: L(f\otimes \chi \Omega_{K},1/2)\neq 0\}\gg_f \begin{cases} |D_K|^{1/1058},& \text{if $f$ is holomorphic,}\\ |D_K|^{1/2648},& \text{if $f$ is Maa{\ss},}\end{cases}$$
where $\Omega_K$ is a Hecke character of $K$ of conductor $1$ and $\infty$-type $\alpha\mapsto (\alpha/|\alpha|)^k$.
\end{cor}
\begin{remark}
We obtain similar results for general square-free levels $N$, see Corollary \ref{cor:n=2}. 
\end{remark}

The case of width $n=3$ is also very appealing as in this case the triple period $\langle f_1f_2f_3,1\rangle$ is related to \emph{triple convolution $L$-functions} via the Ichino--Watson formula \cite{Watson02}, \cite{Ichino08}. This leads to the following non-vanishing result for level $1$ Maa{\ss} forms.
 \begin{cor} \label{cor:nonvanishinglevel1}Let $f_1$ be a fixed Hecke--Maa{\ss} cusp form of level $1$. Then for any $\eps>0$ there exist a constant $c=c(f_1, \eps)>0$ such that for any $T\geq c$ we have for all but $O_\eps (T^{2\eps})$ Hecke--Maa{\ss} cusp forms $f_2$ of level $1$ with $|t_{f_2}-T|\leq T^{\eps}$ that there exists a Hecke--Maa{\ss} cusp form $f_3$ not equal to $f_2$ with $|t_{f_3}-T|\leq T^\eps$ such that the following holds: For any imaginary quadratic field $K$ with $|D_K|\geq c T^{35+\eps}$, we have
\begin{align*} & \# \{ \chi_1, \chi_2\in \widehat{\Cl_{K}}  :\\
&  L(f_1\otimes \chi_1,1/2)L(f_2 \otimes \chi_2,1/2) L(f_3 \otimes \chi_1\chi_2,1/2)L(f_1\otimes f_2\otimes f_3,1/2) \neq 0 \}\\
& \hspace*{+9cm}\gg_{T} |D_K|^{1/1766}.
\end{align*} 
%as $|D_K|\rightarrow \infty$.
\end{cor}
In the case of holomorphic forms, we can show non-vanishing for arbitrary width $n$ (stated here in the simplest case of level $1$, we refer to Corollary \ref{cor:holononvgeneral} for a more general statement).
\begin{cor}\label{cor:nonvanishinglevel1holo}
%(such that $\mathcal{S}_{\kappa}(\Gamma_0(N))$ only contain newforms(?))
Let $n\geq 1$, $k_1,\ldots, k_n\in 2\Z_{>0}$ and put  $k=\sum_i k_i$. For $i=1,\ldots, n$, let $g_i\in \mathcal{S}_{k_i}(1)$ be cuspidal holomorphic Hecke eigenform of level $1$. Then for each $\eps>0$ there exists a constant $c=c(\eps)>0$ such that the following holds: For any imaginary quadratic field $K$ with $|D_K|\geq c k^{45+\eps}$, we have 
\begin{align*} & \# \{ (\chi_1,\ldots, \chi_{n+1})\in \mathbf{Wide}(K, n+1), \text{\rm norm $1$ Hecke eigenforms }g\in \mathcal{S}_{k}(1) : \\
& L(g_1\otimes \chi_1 \Omega_{i,K },1/2)\cdots L(g_{n}\otimes \chi_{n} \Omega_{n,K},1/2)L(g\otimes \chi_{n+1} \Omega_{n+1,K},1/2) \neq 0\}\\
&  \hspace*{+9cm}\gg_{k} |D_K|^{(n+1)/2115}, 
\end{align*} 
where $\Omega_{i,K}$ are Hecke characters of $K$ of $\infty$-type $x\mapsto (x/|x|)^{k_i}$  for $i=1,\ldots, n$, and $\Omega_{n+1,K}=\prod_{i=1}^{n} \Omega_{i,K}$.
\end{cor}
\begin{remark}
Note that it follows in particular that the respective non-vanishing sets in Corollaries \ref{cor:n=2simpel}, \ref{cor:nonvanishinglevel1} and \ref{cor:nonvanishinglevel1holo} are \emph{non-empty} as soon as respectively, $|D_K|\geq c T^{22+\eps}$,   $|D_K|\geq c T^{35+\eps}$ and $|D_K|\geq c k^{45+\eps}$.
\end{remark}

\subsection{Main moment calcualtion}%, which are quantitive in many aspects. 

The above non-vanishing results are all corollaries our main $L$-function calculation. To state this, denote by $\mathcal{B}_{k}^*(N)$ the set of all Hecke--Maa{\ss} newforms of level $N$ and even weight $k\geq 0$ (i.e. raising operators applied to either level $N$ classical Hecke--Maa{\ss} eigenforms or to $y^{k'/2}g$ with $g\in \mathcal{S}_{k'}(N)$ a holomorphic cuspidal newform of even weight $k'\leq k$). Then we have the following moment calculation.

\begin{thm}\label{thm:main}
%For $i=1,\ldots, n$, let $f_i\in A(\Gamma_0(N)\backslash \H,  k_i)$ be a weight $k_i\in 2 \Z$ Hecke new at all finite places such that $\sum \kappa_i=0$.  
Let $N\geq 1$ be a fixed square-free integer and $n\geq 1$. For $i=1,\ldots, n$, let $\pi_i$ be a cuspidal automorphic representation of $\GL_2(\A)$ of conductor $N$ with trivial central character, spectral parameter $t_{\pi_i}$ and even lowest weight $k_{\pi_i}$. Let $k_1,\ldots, k_n\in 2\Z$ be integers such that  $|k_i|\geq k_{\pi_i}$ and $\sum_i k_i=0$.  

Let $|D_K|\rightarrow \infty$ transverse a sequence of discriminant of imaginary quadratic fields $K$ such that all primes dividing $N$ split in $K$. For each $K$, pick Hecke characters $\Omega_{i,K}$ with infinite types $x\mapsto (x/|x|)^{k_i}$ such that $\prod_i \Omega_{i,K}$ is the trivial Hecke character.
%Let $|D_K|\rightarrow \infty$ transverse a sequence of discriminants of imaginary quadratic fields $K$ such that all primes dividing $N$ splits in $K$. For each $K$, pick Hecke characters $\Omega_{i,K}$ with infinite type $x\mapsto (x/|x|)^{k_i}$ so that $\prod_i \Omega_{i,K}$ is the trivial Hecke character. %(notice that this is always possible since, we know that $\prod_i \Omega_{i,K}$ is a class group character).

Then we have for $f_i\in \mathcal{B}^*_{k_i}(N)$ belonging to $\pi_i$ and any $\eps>0$;
\begin{align}
\nonumber&\sum_{(\chi_i)_{1\leq i\leq n}\in \mathbf{Wide}(K,n)}\, \prod_{i=1}^n \eps_{\chi_i,f_i}  c_{f_i} L(\pi_i\otimes \chi_i \Omega_{i,K},1/2)^{1/2}\\
\label{eq:maincomp1}&=  \frac{|\Cl_K|^n}{|D_K|^{n/4}}\Biggr(\langle\prod_{i=1}^n f_i,1\rangle+O_\eps \left( || \prod_{i=1}^n f_i||_2 |D_K|^{-1/16}T^{5/2}n^{15/4} (T|D_K|n)^\eps\right)\Biggr),
\end{align}  
where $T= \max_{i=1,\ldots, n} |k_i|+|t_{\pi_i}|+1$, the weights $\eps_{\chi,f_i}$ are all of norm $1$ and $c_{f_i}$ are certain constants depending only on $f_i$. 
%where $f_i=\otimes f_{i,p}\in \pi_i$ are test vectors new at all finite places and of weight $k_i$ at $\infty$ and $L^2$-normalized with respect to the Petersson norm. The constants $c_{\pi_i,k_i}$ are given by ?? and $\eps_{\chi,i}$ are all of norm $1$. 
%The implied constant is allowed to depend on the level $N$.  

%If the level is $N=1$ we can improve the error term to
%\begin{equation}\label{eq:errorL1}O_\eps \left( || \prod_{i=1}^n f_i||_2 |D_K|^{-1/12 }T^{1/2}n (T|D_K|n)^\eps\right). \end{equation}
%as $D\rightarrow \infty$ with $D$ a fundamental discriminant such that all primes dividing $N$ splits in $\Q[\sqrt{-D}]$, where $|\eps_i|=1$ for $i=1,\ldots, n$.  Here $\lambda_{i,-D}$ is a Hecke character of $\Q[\sqrt{-D}]$ with infinite type $x\mapsto (x/|x|)^{k_i}$ chosen so that $\prod_i \lambda_{i,-D}=1\in \widehat{\Cl_{-D}}$. %N squarefree here. We know that $\prod_i \lambda_i$ is a class group character.
\end{thm}
\begin{remark}We obtain a slightly more general statement that applies to old-forms as well, meaning that we allow for the automorphic representations $\pi_i$ to have different conductors. Furthermore, we obtain an improved error-term in the case of holomorphic forms and/or in the case of level $1$. We refer to Corollary \ref{cor:main} for details (including the exact values of the constants $c_f$). As an application we can also calculate a related {\lq\lq}diagonal wide moment{\rq\rq}, see Corollary \ref{cor:diagonalmoment}.\end{remark}

%\subsubsection{Applications to non-vanishing}\label{sec:applnonv}
% The simplest case is $n=2$ in which case the main term in (\ref{eq:maincomp1}) automatically is non-zero. Thus we obtain the following generalization and strengthening of Michel and Venkatesh \cite[Theorem 1]{MichVenk06} (see Corollary \ref{cor:n=2} below for a version for general level).   

The plan of the paper is as follows. In Section \ref{sec:weaknonv} we will introduce the notion of \emph{weak simultaneous non-vanishing}. Section \ref{sec:background} provides the necessary background on imaginary quadratic fields and automorphic forms. Section \ref{sec:waldspurger} proves an explicit and classical Waldspurger type formula for general weight automorphic forms. In Section \ref{sec:technical} we will prove two technical lemmas; one on the norm of powers of the hyperbolic Laplacian and one on a lower bound for the $L^2$-norm of a product of automorphic forms. In Section \ref{sec:main}, we will prove our main moment calculation. Finally, Section \ref{sec:applications} proves non-vanishing of certain automorphic periods which combined with our moment calculation yields weak simultaneous non-vanishing results.
% \begin{cor} Let $\pi_1$ be a fixed spherical automorphic representation of $\GL_2(\A)$ (i.e. corresponding to a Maa{\ss} cuspform of level $1$). Then for all $\eps>0$ and $T\gg_{\eps} 1$ there exist spherical automorphic representations $\pi_2\neq \pi_3$ (i.e. corresponding to Maa\ss forms for $\SL_2(\Z)$) with $|t_{\pi_2}-T|\ll T^{2\eps-1}$ and $|t_{\pi_3}-T|\ll T^\eps$ such that as $K$ transverses a sequence of imaginary quadratic fields, we have;
%\begin{align*} & \# \{ \chi_1, \chi_2\in \widehat{\Cl_{K}}  :  L(\pi_1\otimes \chi_1,1/2)\cdot L(\pi_2 \otimes \chi_2,1/2)\cdots\\
%& \qquad \qquad \cdots L(\pi_3 \otimes \chi_1\chi_2,1/2)\cdot L(\pi_1\otimes \pi_2\otimes \pi_3,1/2) \neq 0 \} \gg_T |D_K|^{\delta }, 
%\end{align*} 
%as $D_K\rightarrow \infty$.
%\end{cor}
\subsection*{Acknowledgements}
We would like to thank Valentin Blomer for useful feedback on an earlier version of the paper. The author's research was supported by the German Research Foundation under Germany's Excellence Strategy EXC-$2047/1$-$390685813$.
\section{Weak simultaneous non-vanishing}\label{sec:weaknonv} %%%%%%%%%%%%%%%%%%%%%%%%%%%%%%%%
The non-vanishing results proved in the present paper, we name \emph{weak simultaneous non-vanishing}. This terminology is referring to the fact that we show non-vanishing of twists of different $L$-functions with some  {\lq\lq}algebraic dependence{\rq\rq} on the twists (their product is trivial). Ideally, of course we would like to show non-vanishing for the same character, but this seems out of reach with current methods. 
%To compare the non-vanishing results of this paper with the existing literature, we will in this section analyze in which cases weak simultaneous non-vanishing can be deduces. 

Let us start by considering the simplest case $n=2$. This means that we are studying the non-vanishing of two maps $L_1,L_2: G\rightarrow \C$ where $G$ is a finite abelian group. If both $L_1$ and $L_2$ are non-vanishing for more than $50\%$ of $g\in G$, then by the pigeon hole principle there is some $g\in G$ such that $L_1(g)L_2(g)\neq 0$. But clearly one can construct examples where $L_1,L_2$ vanish for exactly $50\%$ of $g\in G$ but there is no simultaneous non-vanishing.

More generally, consider $L_1,\ldots, L_n: G\rightarrow \C$. Then we say that $L_1,\ldots, L_n$ are {\bf weakly simultaneously non-vanishing} if 
$$\{(g_1,\ldots, g_n)\in \mathbf{Wide}(G,n):  L_i(g_i)\neq 0,\text{ for }i=1,\ldots, n\}\neq \emptyset .$$
Recall that by (\ref{eq:Wide}) this means that there exist $g_1,\ldots, g_{n}\in G$ such that 
$$ g_1\cdots g_n=1_G,\quad \text{and}\quad L_1(g_1)\cdots L_{n}(g_{n})\neq 0.  $$ 

This we think of as expressing that one can find non-vanishing for $L_1,\ldots, L_n$ with some {\lq\lq}algebraic dependence{\rq\rq}. This is interesting since most non-vanishing results for automorphic $L$-functions are obtained by using the method of mollification, which gives no information about the algebraic structure of the non-vanishing set. Of course if all of $L_1,\ldots, L_n$ vanish on a very large percentage of elements of $G$ then one gets a weak simultaneous non-vanishing for purely combinatorial reasons. In most cases this is not the case, which we make precise as follows.

\begin{prop}
Let $n\geq 2$ be an integer and $0\leq c\leq 1$. Then there exists a finite abelian group $G$ and maps $L_1,\ldots, L_n:G\rightarrow \C$ satisfying
$$ \#\{ g\in G:  L_i(g)\neq 0 \}\geq c |G|, \quad i=1,\ldots, n,$$  
with \underline{no} weak simultaneous non-vanishing if and only if $c\leq 1/2$.
 \end{prop}
 \begin{proof}
Assume first of all that $c>1/2$. Then if $g_1,\ldots, g_{n-2}$ are such that $L_i(g_i)\neq 0$ for $i=1,\ldots, n-2$. Then again by the pigeon hole principle there is at least one $g\in G$ such that $L_{n-1}(g)\neq 0$ and $L_n((g_1\cdots g_{n-1}g)^{-1})\neq 0$ (since all of the elements $(g_1\cdots g_{n-1}g)^{-1}$ are different as $g\in G$ varies). 

On the other hand if $c\leq 1/2$. Then we can consider any finite abelian group $G$ with a subgroup $H$ of index $2$. Now we let $L_i(g)\neq 0\Leftrightarrow g\in H$ for $i=1,\ldots, n-1$ and let $L_n$ be non-vanishing on the complement of $H$. In this case it is easy to check that there is no weak simultaneous non-vanishing.
% set of $\max (a_n,q/2)$ elements of $G$ containing $H$ and let $L_{n-1}$ be non-vanishing on the complement (notice $c_{n-1}\leq 1/2$). Then it is easy to check that there is no weak simultaneous non-vanishing. (NOT CORRECT!) 
 \end{proof}
 
This shows that one needs to know non-vanishing for at least $50\%$ for some of the maps $L_i$ in order to get weak simultaneous non-vanishing for purely combinatorial reasons. This is very far from being known in the case of the Rankin--Selberg $L$-functions studied in this paper as even a positive proportion of non-vanishing seems out of reach with current methods (see \cite{MichelVenk07} and \cite{Templier11}).

\section{Background}\label{sec:background}
\subsection{Different incarnations of the class group}
Let $K$ be an imaginary quadratic field of discriminant $D<-6$. Denote by $\mathcal{I}_K$ the group of integral fractional ideals of $K$, $\mathcal{P}_K$ the subgroup of principal fractional ideals and $\Cl_K= \mathcal{I}_K/ \mathcal{P}_K$ the class group of $K$, which we know from Gau{\ss} is a finite group. Given a fractional ideal $\mathfrak{a}\in \mathcal{I}_K$, we denote by $[\mathfrak{a}]\in \Cl_K$ the corresponding ideal class. We denote by $[\alpha_1,\alpha_2]$ the ideal generated by $\alpha_1,\alpha_2\in K$ over $\Z$, and by $\widehat{\Cl_K}$ the group of class group characters, i.e. group homomorphisms $\chi: \Cl_K\rightarrow \C^\times$.

Let $\A_K^\times$ resp. $\A_{K,\fin}^\times$  denote the id\'eles resp. finite id\'eles of $K$ and $\widehat{\mathcal{O}_K}^\times=\prod_{p} \mathcal{O}_p^\times$ the standard maximal compact subgroup of $\A_{K,\fin}^\times$. Then we have the natural isomorphisms 
\begin{equation}\label{eq:idelic}\mathcal{I}_K\cong \A_{K,\fin}^\times/\widehat{\mathcal{O}_K}^\times\quad \text{and}\quad \Cl_K\cong K^\times\backslash \A_{K,\fin}^\times/\widehat{\mathcal{O}_K}^\times.\end{equation}
Given $\mathfrak{a}\in \mathcal{I}_K$, we denote by $\hat{\mathfrak{a}}\in \A_{K,\fin}^\times$ any lift of the corresponding element of $ \A_{K,\fin}^\times/\widehat{\mathcal{O}_K}^\times$ under the above isomorphism.
 %and $H_K\cong K^\times\backslash \A_{K,\fin}^\times/\mathcal{O}_K^\times$.    

\subsubsection{Heegner forms} 
We refer to \cite{Darmon94} for a concise treatment of the following material. Let $N$ be a square-free integer such that all primes dividing $N$ splits completely in $K$. Consider a residue class $r$ mod $2N$ such that $r^2\equiv D\modulo 4N$.  
For $(a,b,c) \in \Z^3$ having greatest common divisor equal to $1$ and satisfying $b^2 - 4ac = D$, $a \equiv 0 \pmod{N}$, and $b \equiv r \pmod{2N}$, we denote by $[a,b,c]$ the integral binary quadratic form
\begin{equation} \label{eq:Qxy} Q(x,y)=ax^2+bxy+cy^2. \end{equation}
We call such a quadratic form a {\bf Heegner form of level} $N$ and {\bf orientation $r$} and denote by $\mathcal{Q}_D(N,r)$ the set of all such forms, which carries an action of the Hecke congruence subgroup $\Gamma_0(N)$ via coordinate transformation. It is a well-known facts extending Gau{\ss} that the map $\Gamma_0(N)\backslash\mathcal{Q}_D(N,r)\rightarrow \Cl_K$ defined by 
$$[a,b,c]\mapsto [a, (-b+\sqrt{D})/2],$$ 
is a bijection.

Given a Heegner form $Q=[a,b,c]\in \mathcal{Q}_D(N,r)$ we define the associated {\bf Heegner point} as:
\begin{equation}\label{eq:HeegnerPt}z_Q:=\frac{-b+\sqrt{D}}{2a}\in \Hb.\end{equation}
This defines a map $\mathcal{Q}_D(N,r)\rightarrow \Hb$ which is equivariant with respect to the action $\Gamma_0(N)$ (acting via linear fractional transformation on $\Hb$). In particular, we get a map $\Cl_K\rightarrow \Gamma_0(N)\backslash \Hb$ using the above.
% In particular associated to an ideal class $[\mathfrak{a}]\in \Cl_K$, we get a Heegner point $z_{\mathfrak{a},N,r} \in \Gamma_0(N)\backslash \H$ (using the isomorphism  $\Cl_K\cong \Gamma_0(N)\backslash\mathcal{Q}_D(N,r)$).
\subsubsection{Oriented embeddings} \label{sec:orientedem}
Again let $(a,b,c) \in \Z^3$ have greatest common divisor equal to $1$ and satisfy $b^2 - 4ac = D$, $a \equiv 0 \pmod{N}$, and $b \equiv r \pmod{2N}$. Associated to the tripe $(a,b,c)$ we define an (algebra) embedding $\Psi : K \hookrightarrow \mathrm{Mat}_{2 \times 2}(\Q)$ by
\begin{equation}
\label{eq:Psi}
\Psi(\sqrt{D}) \coloneqq \begin{pmatrix} b & 2c \\ -2a & - b \end{pmatrix}.
\end{equation}
This embedding satisfies
\[\Psi(K) \cap \left\{\begin{psmallmatrix} a & b \\ c & d\end{psmallmatrix} \in \mathrm{Mat}_{2 \times 2}(\Z) :  N| c\right\} = \Psi(\mathcal{O}_K),\]
%in particular, $\Psi(E) \cap \Gamma_0(q) = \Psi(\OO_E^{\times})$
where $\mathcal{O}_K$ denotes the ring of integers of $K$. This means that $\Psi$ is an {\bf optimal embedding of level} $N$ and {\bf orientation $r$}. Conversely, every oriented optimal embedding of level $N$ arises from such a triple of integers $(a,b,c) \in \Z^3$. Denote by $\mathcal{E}_{D}(N,r)$ the set of all such embeddings. The congruence subgroup $\Gamma_0(N)$ acts on $\mathcal{E}_{D}(N,r)$ by conjugation, namely
\[(\gamma \cdot \Psi)(x + \sqrt{D} y) := \gamma^{-1} \Psi(x + \sqrt{D} y) \gamma\]
for $\gamma \in \Gamma_0(N)$. 
%We have that
%\[\begin{pmatrix} D & -B \\ -C & A \end{pmatrix} \begin{pmatrix} x + by & 2cy \\ -2ay & x - by \end{pmatrix} \begin{pmatrix} A & B \\ C & D \end{pmatrix} = \begin{pmatrix} x + (2aAB + b(AD + BC) + 2cCD)y & 2(aB^2 + bBD + cD^2)y \\ -2(aA^2 + bAC + cC^2)y & x - (2aAB + b(AD + BC) + 2cCD)y \end{pmatrix}.\]

There is a natural bijection between oriented optimal embeddings $\Psi$ of level $N$ and orientation $r$ as in \eqref{eq:Psi} and Heegner forms $Q = [a,b,c]$ as in \eqref{eq:Qxy} (since these are both completed determined by $(a,b,c) \in \Z^3$), which is equivariant with respect to the action of $\Gamma_0(N)$.  By the above we have a bijection
\begin{equation} \label{eq:embeddingCl} \Gamma_0(N)\backslash \mathcal{E}_{D}(N,r) \rightarrow \Cl_K \end{equation}

Given an optimal embedding $\Psi$ of level $N$, we can extend it to an (algebra) embedding $\Psi_{\A}:\A_K \rightarrow \mathrm{Mat}_{2\times 2}(\A)$ by tensoring (over $\Q$) by $\A$. The local components of $\Psi_\A$ are defined as follows: If $p$ is a prime of $\Q$ which is inert in $K$ with $p\mathcal{O}_K=\mathfrak{p}$, then $K\otimes \Q_p\cong K_\mathfrak{p}$ and thus we get an embedding $\Psi_p: K_\mathfrak{p}\rightarrow \mathrm{Mat}_{2\times 2}(\Q_p)$ given by  
$$K\otimes \Q_p\ni x\otimes y \mapsto \Psi(x)\otimes y \in \mathrm{Mat}_{2\times 2}(\Q_p), $$ 
defined up to the choice of isomorphism $K\otimes \Q_p\cong K_p$  (similarly for the inert infinite place). If $p$ is ramified with $p\mathcal{O}_K=\mathfrak{p}^2$, then $K\otimes \Q_p\cong K_\mathfrak{p}$ and we get a map $\Psi_p:K_\mathfrak{p}\rightarrow \mathrm{Mat}_{2\times 2}(\Q_p)$ by tensoring as in the inert case. Finally, if $p$ is split in $K$ with $p\mathcal{O}_K=\mathfrak{p}\overline{\mathfrak{p}}$ then we have an algebra isomorphism $K\otimes \Q_p\cong K_\mathfrak{p} \times K_{\overline{\mathfrak{p}}}$ given by
\begin{equation}\label{eq:adelicext}  K\otimes \Q_p\ni j_1x+j_2y\mapsto (x,y)\in K_\mathfrak{p} \times K_{\overline{\mathfrak{p}}}, \quad x,y\in \Q_p,  \end{equation}
where 
$$j_1= (1\otimes 1+\sqrt{D}\otimes (\sqrt{D})^{-1})/2,\quad j_1= (1\otimes 1-\sqrt{D}\otimes (\sqrt{D})^{-1})/2.$$
Here we consider $\sqrt{D}$ as an element of $\Q_p$ and use that $\Q_p\cong K_\mathfrak{p}$ as $p$ splits in $K$. By using this we get an algebra embedding $\Psi_p:K_\mathfrak{p} \times K_{\overline{\mathfrak{p}}}\rightarrow\mathrm{Mat}_{2\times 2}(\Q_p) $ by tensoring. Again this is well-defined up to the choice of isomorphism $\Q_p\cong K_\mathfrak{p}$. 
%To see that  $\Q_p\cong K_\mathfrak{p}$, we use that 

\subsection{Hecke characters of imaginary quadratic fields}\label{sec:Heckechar}
Let $K$ be an imaginary quadratic field of discriminant $D<-6$. In this paper we will be working with Hecke characters of $K$ of conductor $1$, which (in the classical picture) are unitary characters $\chi: \mathcal{I}_K\rightarrow \C^\times$ such that for $(\alpha)\in \mathcal{P}_K$, we have $\chi((\alpha))=\chi_\infty^{-1}(\alpha)$ for some character $\chi_\infty:\C^\times\rightarrow \C^\times$, which we call the $\infty$-type of $\chi$. By considering the induced representation one can see that given $\chi_\infty$ such that $\chi_\infty(-1)=1$, we have exactly $|\Cl_K|$ Hecke characters with $\infty$-type $\chi_\infty$; if $\chi_0$ is any Hecke character with $\infty$-type $\chi_\infty$ then the set of all such Hecke characters is given by $\{\chi_0\chi:  \chi\in \widehat{\Cl_K}\}$. We will only be considering the $\infty$-types $\alpha\mapsto (\alpha/|\alpha|)^k$ for $k\in 2\Z$.

Given a Hecke character $\chi$ as above with $\infty$-type $\chi_\infty$, we get using the isomorphism (\ref{eq:idelic}) an (id\'elic) Hecke character 
$$\Omega: K^\times \backslash \A_{K}^\times/\mathcal{O}_K^\times\rightarrow \C^\times.   $$
The above conditions translates to the fact that $\Omega$ is unramified at all finite places of $K$ and the $\infty$-component $\Omega_\infty$ is equal to $\chi_\infty$. 

Associated to a Hecke character $\chi$ as above with $\infty$-type $\alpha\mapsto (\alpha/|\alpha|)^k$, there is a theta series 
$$ \theta_\chi(z):= \sum_{\mathfrak{a}\text{ int. ideal of $\mathcal{O}_K$}} e^{2\pi i(\mathrm{N}\mathfrak{a}) z}(\mathrm{N}\mathfrak{a})^{k/2}\chi(\mathfrak{a})\in \mathcal{M}_{k+1}(\Gamma_0(|D|), \chi_K),$$ 
which is a modular form of weight $k+1$, level $|D|$ and nebentypus equal to the quadratic character $\chi_K$ associated to $K$ via class field theory. Furthermore, we know that $\theta_\chi$ is non-cuspidal exactly if $k=0$ and $\chi$ is a genus character of the class group of $K$ (see \cite[Theorem 12.5]{Iw2}). Recall that this is an example of automorphic induction from $\GL_1/K$ to $\GL_2/\Q$. 
%\begin{equation} \theta_\chi(z):=\end{equation}

\subsection{Automorphic forms} \label{sec:autoforms} 
In what follows we follow \cite[Chapters 2+3]{Bump97}. Let $L^2(\Gamma_0(N), k)$ denote the $L^2$-space of automorphic functions of level $N$ and weight $k\in 2\Z$. That is, measurable maps $f :\Hb\rightarrow \C$ satisfying:
\begin{itemize}
\item The \emph{automorphic condition} of weight $k$ and level $N$:
$$ f(\gamma z)=j_\gamma(z)^k f (z), $$
for all $\gamma=\begin{psmallmatrix} a &b \\ c& d \end{psmallmatrix}\in \Gamma_0(N)$, where 
$$j_\gamma(z):= j(\gamma,z)/|j(\gamma, z)|,\quad j(\gamma, z)=cz+d,$$
and 
$$  \Gamma_0(N):= \{\begin{psmallmatrix} a &b \\ c& d \end{psmallmatrix}\in \PSL_2(\Z):  N|c\}.  $$
\item The \emph{$L^2$-condition}:
$$ ||f ||_2^2:=\langle f ,f \rangle= \int_{\Gamma_0(N)\backslash \Hb} |f (z)|^2 d\mu (z)<\infty,$$
where $d\mu (z)= y^{-2}dxdy$ and $\langle\cdot, \cdot \rangle$ is the Petersson inner-product. Notice that the above integral is well-defined since $|j_\gamma(z)|=1$.
\end{itemize}
We have the weight $k$ {\bf raising} and {\bf lowering operators} acting on $C^\infty(\Hb)$ (the space of smooth functions on $\Hb$) given by
$$  R_k=(z-\overline{z})\frac{\partial}{\partial z}+\frac{k}{2} ,\quad  L_k=-(z-\overline{z})\frac{\partial}{\partial \overline{z}}-\frac{k}{2}.$$
They define maps 
\begin{align*}  &R_k: L^2(\Gamma_0(N), k)\cap C^\infty(\Hb) \rightarrow L^2(\Gamma_0(N), k+2)\cap C^\infty(\Hb), \\
&L_k: L^2(\Gamma_0(N), k)\cap C^\infty(\Hb) \rightarrow L^2(\Gamma_0(N), k-2)\cap C^\infty(\Hb),\end{align*}
which are adjoint in the following sense
\begin{align}\label{eq:adjoint}
\langle R_k f _1, f _2\rangle= - \langle  f _1, L_{k+2}f _2\rangle, 
\end{align} 
for $f _1\in L^2(\Gamma_0(N), k)\cap C^\infty(\Hb)$ and $f _2\in L^2(\Gamma_0(N), k+2)\cap C^\infty(\Hb)$. Furthermore, we have the following product rule
$$  R_{k_1+k_2} (f _1f _2)= (R_{k_1} f _1) f _2+ f _1(R_{k_2} f _2),  $$
for $f _i \in  L^2(\Gamma_0(N), k_i)\cap C^\infty(\Hb)$, and similarly for the lowering operator.

The weight $k$ {\bf Laplacian} acting on $L^2(\Gamma_0(N), k)\cap C^\infty(\Hb) $ is defined as
$$ \Delta_k=-R_{k-2} L_{k}+\lambda(k/2) =-L_{k+2} R_{k}+\lambda(-k/2) ,$$
where $\lambda(s)=s(1-s)$. This defines a symmetric, unbounded operator on $L^2(\Gamma_0(N), k)$ with a unique self-adjoint extension which we also denote by $\Delta_k$ with domain $D(\Delta_k)$ (suppressing the level $N$ in the notation). 

A {\bf Maa{\ss} form} of weight $k$ and level $N$ is a (necessarily real analytic) eigenfunction of $\Delta_k$. Given a Maa{\ss} form $f $ of eigenvalue $\lambda$ we denote by $t_f :=\sqrt{\lambda-1/4}$ the {\bf spectral parameter} of $f $ (if $\lambda>1/4$ we always pick the positive square root).  

Denote by $\mathcal{S}_{k}(N)$ the vector space of weight $k$ and level $N$ (classical) holomorphic cusp forms. If $g\in\mathcal{S}_{k}(N)$, then it is easy to see that $y^{k/2}g$ is a Maa{\ss} form of weight $k$ and level $N$ of eigenvalue $\lambda(k/2)$. In fact, it can be show that any Maa{\ss} form of weight $k\geq 0$ and level $N$ is of the form
$$ R_{k-2} \cdots R_{k_0} y^{k_0/2} g, \quad \text{with } g\in \mathcal{S}_{k_0}(N) \text{ where }k_0\leq k\text{ and } k_0\equiv k\modulo 2 $$   
or 
$$R_{k-2}\cdots   R_0 f ,\quad   \text{with $f $ a Maa{\ss} form of weight $0$ and level $N$}.$$
And similarly for $k<0$ now with lowering operators and anti-holomorphic cusp forms.

Furthermore, we say that a Maa{\ss} form of weight $k$ and level $N$ is a {\bf Hecke--Maa{\ss} eigenform} if it is an eigenfunction for the Hecke operators $T_n$ with $(N,n)=1$ (which compute with the action of the raising and lowering operators) and the reflection operator 
$$X:L^2(\Gamma_0(N),k)\rightarrow L^2(\Gamma_0(N),k),\quad (Xf )(z):=f (-\overline{z}).$$
Finally, we say that a Hecke--Maa{\ss} eigenform is a {\bf Hecke--Maa{\ss} newform} if it is an eigenfunction for {\it all} Hecke operators $T_n$ with $n\geq 1$. 

Denote by $ \mathcal{B}^*_{k,\mathrm{hol}}(N)$ the set of all  $y^{k/2} g$, where $g\in \mathcal{S}_{k}(N)$ is an $L^2$-normalized Hecke newform (note that this means that $y^{k/2} g$ is a Hecke--Maa{\ss} newform of weight $k$ and level $N$), and by $ \mathcal{B}^*(N)$ the set of all {\it non-constant} $L^2$-normalized Hecke--Maa{\ss} newforms of weight $0$ and level $N$ (we will sometimes refer to these simply as (classical) {\lq\lq}Maa{\ss} forms{\rq\rq}). Then it follows from Atkin--Lehner theory that we have the following (normal) basis consisting of Hecke--Maa{\ss} eigenforms for the subspace of $L^2(\Gamma_0(N),k)$ spanned by {\it non-constant} Maa{\ss} forms of weight $k\geq 0$ and level $N$:
\begin{align} \label{eq:basis} & \mathcal{B}_{k}(N):=\bigcup_{dN'|N} \nu^*_{d,N'} R_{k-2}\cdots R_0\mathcal{B}^*(N') \\
\nonumber&\qquad \cup \bigcup_{dN'|N}\bigcup_{\substack{0<k_0\leq k,\\ k_0\equiv k\modulo 2}} \nu^*_{d,N'} R_{k-2}\cdots R_{k_0}\mathcal{B}^*_{k_0,\mathrm{hol}}(N'),\end{align}
where $ \nu^*_{d,N'}: L^2(\Gamma_0(N'), k)\rightarrow L^2(\Gamma_0(N), k)$ are defined by $( \nu^*_{d,N'}f )(z):=f (dz)$. If $k<0$ we have a similar basis now with lowering operators and anti-holomorphic cusp forms.

%Thus we can a basis $\Delta_k$  
%We define the following (what about old forms??)
%$$ \mathcal{B}_{k,\mathrm{hol}}(N)=\bigcup_{\substack{0<k_0\leq k\\ k_0\equiv k \modulo 2}}\{R_{k-2} \cdots R_{k_0} F:   f\in \mathcal{S}_{k_0}(\Gamma_0(N))\text{ $L^2$-normalized Hecke--Maa{\ss}form}\} 
%and
%$$ \mathcal{B}_{k,\mathrm{Maa{\ss}}}(N)=\{R_{k-2} \cdots R_0 f  :   f\in C^\infty(\Gamma_0(N)\backslash \H) \text{ $L^2$-normalized Hecke--Maa{\ss}form}\}  $$

Using (\ref{eq:adjoint}) one sees that for any $f  \in\mathcal{B}_{k}(N)$, we have the following useful relation: 
\begin{equation} \label{eq:normraising} ||  R_{k+2l} \cdots R_{k} f   ||_2^2= ||f ||_2^2    \prod_{j=0}^l \left(\frac{k+2j-1}{2}+it_f \right)\left(\frac{k+2j-1}{2}-it_f \right)   \end{equation}
\subsubsection{Ad\'elization of Maa{\ss} forms} Given an element of $f\in L^2(\Gamma_0(N), k)$ we define a lift $\tilde{f}: \GL_2^+(\R)\rightarrow \C$ as 
$$ \tilde{f}(g):= j_g(i)^{-k} f(gi), $$
which satisfies 
$$ \tilde{f}(gk_\theta)=e^{ik\theta} \tilde{f}(g), $$
for all $\theta\in[0,2\pi) $ where $k_\theta=\begin{psmallmatrix} \cos \theta & \sin \theta \\ -\sin \theta & \cos \theta \end{psmallmatrix}$ and $g\in \GL_2^+(\R)$.

Now consider the following decomposition of $\GL_2(\A)$ coming from strong approximation;
\begin{equation} \GL_2(\A)= \GL_2(\Q) K_0(N) \GL_2^+(\R),\end{equation}
where $\GL_2(\Q)$ is embedded diagonally and 
$$K_0(N) := \{k \in \GL_2(\A) :  k_\infty=1, k_p=\begin{psmallmatrix} a_p & b_p\\ c_p & d_p \end{psmallmatrix}\in \GL_2(\Z_p), c_p\in p^r \Z_p, p^r|| N  \}. $$ 
Now we define the {\bf ad\'elization} of $f$ as 
$$ \phi_f(g)=\phi_f(\gamma k g_\infty ):=\tilde{f}(g_\infty), $$
which does not depend on the choice of decomposition 
$$g=\gamma k g_\infty\in  \GL_2(\Q) K_0(N) \GL_2^+(\R).$$

Given a Hecke--Maa{\ss} newform $f$, the ad\'elization $\phi_f$ generates a unique cuspidal automorphic representation $\pi_f=\pi$ of $\GL_2(\A)$. The infinity component of this representation $\pi_\infty$ is a discrete series representation of lowest weight $k_\pi=k$ if $f$ corresponds to a holomorphic Hecke newform of weight $k$. On the other hand if $f$ is of weight $0$ and non-constant (i.e. corresponds to a classical Maa{\ss} form), then $\pi_\infty$ is a principal series representation og lowest weight $k_\pi=0$. We denote by $t_\pi$ the spectral parameter $t_f$ of $f$.

\subsubsection{Automorphic $L$-functions}In general, associated to an automorphic representation $\pi$ of $\GL_n(\A)$ one can define the (finite part of the) $L$-function $L(\pi,s)$ as a product  over finite primes in terms of the Satake parameters and a completed version $\Lambda(\pi,s)$ satisfying a functional equation $\Lambda(\pi,s)= \eps_\pi \Lambda(\check{\pi},1-s)$, where $\eps_\pi$ is of norm $1$ (the root number) and $\check{\pi}$ is the  contragredient of $\pi$. We refer to \cite{GodementJacquet72} for details. Furthermore, given automorphic representations  $\pi_1,\pi_2,\pi_3$ of $\GL_n(\A)$ we will be interested in the Rankin--Selberg convolution $L$-function $L(\pi_1\otimes \pi_2, s)$ (see \cite{JaPiShSha83}), the symmetric square $L$-function $L(\sym^2 \pi_1, s)$ (see \cite[Chapter 3.8]{Bump97}) and the triple convolution $L$-function $L(\pi_1\otimes \pi_2\otimes \pi_3,s )$ (see \cite{Watson02}).   
% or a complementary series representation (Selberg's eigenvalue conjecture predicts that no such exists).
 
\section{A classical version of Waldspurger's formula} \label{sec:waldspurger}%%%%%%%%%%

In order to make our moment calculations explicit, we will need an explicit version of Waldspurger's formula as developed my Martin and Whitehouse \cite{MaWh09} and furthermore translate this to a classical formula. In doing so, we will follow Popa \cite[Chapter 5]{Popa06}. 

\subsection{A formula of Martin and Whitehouse (following Waldspurger)}
Let $\pi$ be an automorphic representation of $\GL_2(\Q)$ of square-free conductor $N$ and even lowest weight $k_\pi$ corresponding to the classical cuspidal newform $f$ (Maa{\ss} or holomorphic also of weight $k_\pi$). Let $D<-6$ be a negative fundamental discriminant with $(D,2N)=1$ and such that all primes dividing $q$ splits in $K=\Q[\sqrt{D}]$. Let $k\geq k_\pi$ be even and let $\Omega:K^\times \backslash \A_K^\times\rightarrow \C^\times$ be an id\'elic Hecke character of conductor $1$ and $\infty$-type $\Omega_\infty(\alpha)=(\alpha/|\alpha |)^k$. Recall from Subsection \ref{sec:Heckechar} that any two such characters differ by a class group character and thus there are $|\Cl_K|$ such characters. 

We will be interested in obtaining an explicit formula in terms of Heegner points of the central value of the Rankin--Selberg $L$-function $L(\pi\otimes \Omega,1/2)$ (by which we mean the Rankin--Selberg convolution of the base change $\pi_K$ of $\pi$ to $\GL_2(\A_K)$ and the automorphic representation $\pi_\Omega$ of $\GL_1(\A_K)$ corresponding to $\Omega$). We note that the above (Heegner) conditions on $D$ and $N$ imply that the root number of $L(\pi\otimes \Omega,s)$ is equal to $+1$.

Let $\Psi_\A:\A_K\hookrightarrow \GL_2(\A)$ be an optimal algebra embedding of level $N$. Then associated to the triple $(\pi,\Omega, \Psi_\A)$, Martin and Whitehouse \cite[Theorem 4.1]{MaWh09} define a specific test vector $\phi_{\mathrm{MW}}\in \pi$ such that we have the following formula:
\begin{align} \label{eq:WaldspurgerMW}
 \frac{\left|\int_{\A^\times K^\times \backslash \A_K^\times} \phi_{\mathrm{MW}}(\Psi_\A(x)) \Omega^{-1}(x) dx\right|^2}{\int_{Z(\A)\GL_2(\Q)\backslash \GL_2(\A)} |\phi_{\mathrm{MW}}(g)|^2dg}
= \frac{L(\pi\otimes \Omega,1/2)}{L(\sym^2 \pi, 1)} \frac{c_\infty(\pi_\infty, k)}{2\sqrt{|D|}} \prod_{p| N}\left(1-\frac{1}{p}\right)^{-1} ,  
\end{align} 
where the measure $dg$ is normalized so that the volume of $Z(\A)\GL_2(\Q)\backslash \GL_2(\A)$ is $\frac{\pi}{3}\prod_{p| N} (1-p^{-2})$ (here we are using that the Tamagawa number of $\GL_2/\Q$ is $2$) and $dx$ is normalized so that $\A^\times K^\times \backslash \A_K^\times$ has volume $2\Lambda(\chi_{K},1)$, where $\chi_K$ is the quadratic character associated to $K$ via class field theory and 
$$\Lambda(\chi_{K},s)= \pi^{-(s+1)/2}\Gamma((s+1)/2)L(\chi_K,s).$$
The local constants are given by:
\begin{align*}c_\infty(\pi_\infty, k)=\begin{cases} (2\pi)^{k} \prod_{j=0}^{k/2-1}(1/4+(t_\pi)^2+j(j+1) )^{-1} & \text{if $\pi_\infty$ is a p.s,}\\
  (2\pi)^{k-k_\pi}\Gamma(k_\pi+1)/\left(\Gamma\left(\frac{k+2}{2}\right)B\left(\frac{k+k_\pi}{2}, \frac{k-k_\pi+2}{2}\right)\right)& \text{if $\pi_\infty$ is a d.s,}\end{cases}  \end{align*}
where {\lq\lq}p.s{\rq\rq} and {\lq\lq}d.s{\rq\rq} refer to respectively, {\lq\lq}principal series{\rq\rq} and {\lq\lq}discrete series{\rq\rq}, and $B(x,y)$ denotes the Beta function.

To make this formula explicit, we need to specify an embedding $\Psi_\A$. To do this, let $[a,b,c]$ be a Heegner form of level $N$ and orientation $r$ and consider the associated optimal embedding $\Psi:K\hookrightarrow  \mathrm{Mat}_{2 \times 2}(\Q)$ of level $N$ (as in Subsection \ref{sec:orientedem}) satisfying
$$ \Psi(K)\cap M_0(N)= \Psi(\mathcal{O}_K), $$ 
where $M_0(N)= \{\begin{psmallmatrix} a&b\\ c&d\end{psmallmatrix}\in \mathrm{Mat}_{2 \times 2}(\Z):  N|c \}$. As described in Subsection \ref{sec:orientedem}, we get by tensoring with $\A$ an associated embedding $\Psi_\A:\A_K^\times \rightarrow \GL_2(\A)$. We write $\Psi_{\fin}$ for the finite component and $\Psi_\infty$ for the infinite component of this embedding.

Now the recipe described in \cite[Chapter 4.2]{MaWh09} gives the following characterization of the test vector $\phi_\mathrm{MW}$; the finite component $\phi_{\mathrm{MW},p}$ at a finite prime $p<\infty$ is uniquely determined (up to scaling) by the invariance under a certain Eichler order, which in our setting is exactly the order in $\GL_2(\Q_p)$ of reduced discriminant $p^{\nu_p(N)}$ (using that $\Psi$ is optimal of level $N$). This means that we can pick $\phi_{\mathrm{MW},p}=\phi_{f,p}=\phi_{f_k,p}$, where $\phi_f$ (resp. $\phi_{f_k}$) are the lifts to $\GL_2(\A)$ of the Hecke--Maa{\ss} newform $f\in L^2(\Gamma_0(N), k_\pi)$ corresponding to $\pi$ (resp. $f_k=R_{k-2}\cdots R_{k_\pi} f$).

At the infinite place the test vector $\phi_{\mathrm{MW},\infty}$ is characterized by being the vector of the minimal $K$-type (in the sense of \cite{Popa08}) such that  
$$\pi_\infty(x)\phi_{\mathrm{MW},\infty}=\Omega_\infty(x)\phi_{\mathrm{MW},\infty}, $$
for all $x\in \Psi_\infty(S^1)\cap \OO_2(\R)$, where $S^1=\{z\in \C^\times : |z|=1\}$ is the maximal compact of $\C^\times$ and $\OO_2(\R)$ is the maximal compact of $\GL_2(\R)$. There is a slight complication due to the fact that the embedding $\Psi_\infty$ defined above does not send the maximal compact $S^1\subset \C^\times$ to $\SO_2(\R)$.  One can however easily check that this is the case after conjugating by 
\begin{equation}\label{eq:gammainfty} \gamma_\infty=\begin{pmatrix} \sqrt{D} & -b \\ 0 & a\end{pmatrix}.\end{equation}
Thus we conclude that the following vector satisfies the conditions specified by Martin and Whitehouse:
$$\phi_{\mathrm{MW},\infty}= \pi(\gamma_\infty) \phi_{f_k,\infty},$$
where $\phi_{f_k,\infty} = R_{k-2}\cdots R_{k_\pi} \phi_{f,\infty}$ is a weight $k$ vector in the representation space $\pi_\infty$. We conclude that we can pick the global test vector as follows:
$$\phi_{\mathrm{MW}}=\pi(\gamma_\infty)\phi_{f_k},$$
where again $f_k=R_{k-2}\cdots R_{k_\pi} f$ and $\gamma_\infty\in\GL_2(\R)\subset \GL_2(\A)$ as in (\ref{eq:gammainfty}). 

%In what comes it will be inconvenient that $\Psi_\infty$ does not send the maximal compact $S^1=\{z\in \C^\times :  |z|=1\}$ of $\C^\times$ to $\SO_2(\R)$. We will fix this as follows: Associated to an optimal embedding $\Psi$ of level $N$ corresponding to $(a,b,c)$, we define 
%\begin{equation} \gamma_\Psi=\begin{pmatrix} \sqrt{D} & -b \\ 0 & a\end{pmatrix},\end{equation}
%and we see that 
%\begin{equation}\label{gamma}\gamma_\Psi^{-1}\Psi_\infty \gamma_\Psi(S^1)= \SO_2(\R). \end{equation}

%(CHECK THAT $\R^\times$ IS INDEED MAPPED DIAGONALLY) where we identify 
%$$\Cl_K= \mathcal{I}/\mathcal{P}\cong K^\times \backslash \A_{K,\fin}^\times/\widehat{\mathcal{O}_K}^\times,$$ 
%and thus we can think of $\mathfrak{a}\in \mathcal{I}\cong  \A_{K,\fin}^\times/\widehat{\mathcal{O}_K}^\times$ as a fractional ideal (resp.) element of $ \A_{K,\fin}^\times/\widehat{\mathcal{O}_K}^\times$ representing the fractional ideal class $[\mathfrak{a}]$. 

For $\phi_{\mathrm{MW}}$ as above, we have for $x_{\fin}\in \A_{K,\fin}^\times$ and $x_\infty\in \C^\times$ that 
$$\phi_{\mathrm{MW}}( \Psi_\infty(x_\infty)\Psi_{\fin}(x_{\fin}))\Omega^{-1}((x_\infty,x_{\fin})),$$
is independent of $x_\infty$. In particular, we get a well-defined map
$$  \Cl_K\ni [\mathfrak{a}]\mapsto  \phi_{\mathrm{MW}}(\Psi_{\A}(\hat{\mathfrak{a}}))\Omega^{-1}(\hat{\mathfrak{a}}),  $$
where $\hat{\mathfrak{a}}\in \A_{K,\mathrm{fin}}^\times$ is any lift of $\mathfrak{a}$ under the first isomorphism in (\ref{eq:idelic}). By the second isomorphism in (\ref{eq:idelic}) it follows that we have a bijection
\begin{equation}\label{decomp}  K^\times \A^\times \backslash \A_K^\times/ \widehat{\mathcal{O}_K}^\times\xrightarrow{\sim} \bigsqcup_{[\mathfrak{a}]\in \Cl_K} \C^\times /\R^\times,  \end{equation}
from which we conclude that
\begin{equation}\label{eq:periodexplicit}
 \int_{\A^\times K^\times \backslash \A_K^\times} \phi_{\mathrm{MW}}(\Psi_\A(x))\Omega^{-1}(x) dx
 =\frac{2}{|D|^{1/2}}\sum_{[\mathfrak{a}]\in \Cl_K}  \phi_{f_k}(\Psi_{\fin}(\hat{\mathfrak{a}})\gamma_\infty)\overline{\Omega(\hat{\mathfrak{a}})}. \end{equation}
 Here one can check the normalization by letting $\phi_{\mathrm{MW}}$ and $\Omega$ being constants and recalling that the total measure of $\A^\times K^\times \backslash \A_K^\times$ is $2\Lambda(\chi_K,1)= 2|\Cl_K| |D|^{-1/2}$ by the class number formula.
\subsection{Explicit representatives of the class group} %%%%%%%%%%%%%%
Consider integral prime ideals $\mathfrak{p}_1=(1), \mathfrak{p}_2,\ldots, \mathfrak{p}_h$ which are representatives for the class group $\Cl_K$ dividing the rational primes $p_i$ which we assume are coprime to $2Na$ (so that $h=|\Cl_K|$ and $p_i\mathcal{O}_K=\mathfrak{p}_i\overline{\mathfrak{p}_i}$ splits in $K$ for $i=2,\ldots, h$). The ideal class $[\mathfrak{p}_i]$ is represented by the id\'ele $\hat{\mathfrak{p}_i}:=(p_i)_{\mathfrak{p}_i}\in \A_K^\times$ (where the subscript means that the element is concentrated at the place $\mathfrak{p}_i$). Thus we see using the definition (\ref{eq:adelicext}) of $\Psi_\A$ that since 
$$j_1\cdot  p_i+ j_2 \cdot 1= 1\otimes \frac{p_i+1}{2}+ \sqrt{D}\otimes \frac{p_i-1}{2\sqrt{D}}\in K\otimes \Q_{p_i},$$ 
we have that 
$$ \Psi_\A((p_i)_{\mathfrak{p}_i})=\begin{pmatrix}\frac{p_i+1}{2}+b\frac{p_i-1}{2\sqrt{D}} & c\frac{p_i-1}{\sqrt{D}} \\ -a\frac{p_i-1}{\sqrt{D}} & \frac{p_i+1}{2}-b\frac{p_i-1}{2\sqrt{D}}\end{pmatrix}_{p_i}.  $$
For $i=2,\ldots, h$, it is a short computation that for an integer $b_i$ with $b_i\equiv b \modulo 2a$ and $b_i^2\equiv D \modulo p_i$ (and put also $b_1=1$ for completeness), we have
\begin{equation}\label{eq:p_i}\mathfrak{p_i}=[(-b_i+\sqrt{D})/2, p_i ].\end{equation} 
Using the congruences for $b_i$ one can show that there is $k_i\in K_0(N)$  such that 		 
$$  \Psi_\A((p_i)_{\mathfrak{p}_i})= \gamma_i k_i (\gamma_i^{-1})_\infty $$
%Then we get for $i=2,\ldots, h$ the following representation in terms of (?) above;
%$$  \Psi_\fin (\mathfrak{p}_i)= \gamma_i (\gamma_i^{-1})_\infty, $$  
with $\gamma_i\in M_2(\Q)$ given by
$$\gamma_i=\begin{pmatrix} p_i & (b_i-b)/2a \\ 0& 1  \end{pmatrix}.$$
 Thus we conclude by the definition of ad\'{e}lization that 
 $$   \phi_{f_k}(\Psi_{\fin}(\hat{\mathfrak{p}_i})\gamma_\infty)= j_{\gamma_i^{-1}\gamma_\infty}(i)^k f_k(\gamma_i^{-1}\gamma_\infty i)= f_k\left(\frac{-b_i+\sqrt{D}}{2ap_i}\right). $$
%In order to make our moment calculations explicit, we will need an explicit version of Waldspurger's formula as developed my Martin and Whitehouse and furthermore translate this to a classical formula. In doing so, we will follow Popa. Let $\pi$ be an automorphic representation of $\GL_2(\Q)$ of square-free conductor $q$ and minimal weight (IS THIS CORRECT LINGO) $k_\pi$. Let$-D<0$ be a negative fundamental discriminant such that all primes dividing $q$ splits in $K=\Q[\sqrt{-D}]$. Let $\chi\in \widehat{\Cl_{-D}}$ we a class group chariacter and let $\Omega:K^\times \backslash \A_K^\times\rightarrow \C^\times$ be the id\'elic lift. Let $k\geq k_\pi$ be an integer let $\lambda: K^\times \backslash \A_K^\times\rightarrow \C^\times$ be a Hecke character with infinite type $\lambda_\infty(x)=\left(\frac{x}{|x|}\right)^k$ with $k\in 2\Z$. There are $h_{-D}$ such characters (see Iwaniec--Kowalski example 4 chapter 3). 
%We will be interested in obtaining an explicit formula in terms of Heegner points of the central value of the Rankin--Selberg $L$-function $L(\pi_K\otimes \Omega,1/2)$ (here $\pi_K$ is the base change of $\pi$ to $\GL_2(\A_K)$).
To proceed we will need to understand how the Heegner points $\frac{-b_i+\sqrt{-D}}{2ap_i}$ behaves as $i=1,\ldots, h$ varies. Let $I:\Gamma_0(N)\backslash\mathcal{E}_{D}(N,r) \rightarrow \Cl_K $ be the bijection in (\ref{eq:embeddingCl}). Then we have the following adaption of \cite[Proposition 6.2.2]{Popa06}
\begin{lemma}\label{lem:Popa}
%Consider the map $E: \Cl_K\rightarrow \mathcal{E}_N/\Gamma_0(N)$ defined by $\mathfrak{p}_i\mapsto \gamma_i^{-1}\Psi \gamma_i$. Then we have
%$$ I(E([\mathfrak{p}_i]))=[I(\psi)]\cdot [\mathfrak{p}_i],  $$
%where $I:\mathcal{E}_N/\Gamma_0(N) \rightarrow \Cl_K $ is the map described in Section \ref{sec:orientedem}. 
We have
$$\gamma_i^{-1} \gamma_\infty i=z_{Q_{\Psi, i}} \in \Hb,$$
where $z_{Q_{\Psi,i}}$ is the Heegner point of a Heegner form $Q_{\Psi, i}$ of level $N$ and orientation $r$ (depending on $\Psi$ and $i$) belonging to the class $I([\psi])\cdot [\mathfrak{p}_i]\in \Cl_K$.
\end{lemma}
\begin{proof}
Consider the binary quadratic form:
$$Q(x,y)=ap_i x^2+ b_i xy+c_iy^2,$$
where 
$$c_i= \frac{b_i^2-D}{4ap_i}$$
is an integer by the above congruence conditions. This means that $Q$ is a discriminant $D$ Heegner form of level $N$ and orientation $r$, with corresponding Heegner point given by
$$ \frac{-b_i+\sqrt{D}}{2ap_i}.$$
Thus the lemma reduces to showing the following identity of ideals (modulo principal ideals): 
\begin{align}
[ap_i, (-b_i+\sqrt{D})/2]=[(-b_i+\sqrt{D})/2, p_i ] \cdot [(-b+\sqrt{D})/2, a].
\end{align}
This follows (as in the proof of \cite[Proposition 6.2.2]{Popa06}) since both sides have the same ideal norm and one can check using the congruence condition on $b_i$ that the right hand side is contained in the left hand side. % Have to check for all products of generators: Only two non-trivial case: $p_i(-b+\sqrt{D})/2=p_i(-b_i+2an +\sqrt{D})/2= ap_i n+ p_i (-b_i+\sqrt{D})/2 \in LHS$, and 
%$(-b+\sqrt{D})/2 \cdot (-b_i+\sqrt{D})/2= (bb_i+D-(b+b_i)\sqrt{D})/4= (bb_i+b_i^2-4ap_ic_i-(b+b_i)\sqrt{D})/4=-ap_ic_i +(b_i-b)/2\cdot (-b_i+\sqrt{D})/2-b_i (-b_i+\sqrt{D})/2 $.
\end{proof}
This implies that the automorphic period (\ref{eq:periodexplicit}) depends on the choice of optimal embedding $\Psi$ but only up to a phase. In particular, the absolute square does not depend on the choice of $\Psi$ as should be the case by (\ref{eq:WaldspurgerMW}).   
%$$\int_{\A^\times K^\times \backslash \A_K^\times} \phi(x) \Omega^{-1}(x) dx,$$
\subsection{An explicit formula}
%Using Lemma \ref{lem:Popa}, we see that if $\hat{\mathfrak{a}}\in\A_K^\times$ is any 
To simplify matters we from now pick our optimal embedding $\Psi$ such that $[a,b,c]$ corresponds to the trivial element of $\Cl_K$ and to lighten notation we write
\begin{equation}\label{eq:Qi}Q_i=ap_i x^2+ b_i xy+c_iy^2,\quad i=1,\ldots, h,\end{equation}
with $p_i$ and $b_i$ as above. Now if $Q\in \mathcal{Q}_D(N,r)$ is any quadratic form such that $[Q]=[\mathfrak{p}_i]$, then it follows from Lemma \ref{lem:Popa} that there is some $\gamma_{Q}\in \Gamma_0(N)$ such that $  z_Q=\gamma_{Q} z_{Q_i}$, which implies that
$$  f_k(z_Q)=j_{\gamma_{Q}}(z_{Q_i})^k f(z_{Q_i})=\Omega_\infty(\alpha_{Q})  \phi_{f_k}(\Psi_{\fin}(\hat{\mathfrak{p}_i})\gamma_\infty), $$
where $\alpha_{Q}= j(\gamma_{Q},z_{Q_i})\in K^\times$. Similarly if $\mathfrak{a}\in \mathcal{I}_K$ is a different representative of the ideal class $[\mathfrak{p}_i]\in \Cl_K$, then we have 
$$\Omega^{-1}(\hat{\mathfrak{a}})=\Omega_\infty(\alpha_{\mathfrak{a}})\Omega^{-1}(\hat{\mathfrak{p}_i}),$$
for some $\alpha_{\mathfrak{a}}\in K^\times$.

From this we conclude by combining (\ref{eq:periodexplicit}) and Lemma \ref{lem:Popa} that 
\begin{equation}\label{eq:MWprelim}  \int_{\A^\times K^\times \backslash \A_K^\times} \phi_{\mathrm{MW}}(\Psi_\A(x))\Omega^{-1}(x) dx
 =\sum_{[Q]\in \Gamma_0(N)\backslash \mathcal{Q}_D(N,r)} f_k(z_{Q})   \overline{\Omega(\widehat{\mathfrak{a}_Q})}\Omega_\infty(  \alpha_{Q, \mathfrak{a}_Q}),    \end{equation}
 where $z_{Q}$ is the Heegner point associated to the Heegner form $Q\in\mathcal{Q}_D(N,r)$, $[\mathfrak{a}_Q]=[Q]$ (under the bijection $\Gamma_0(N)\backslash \mathcal{Q}_D(N,r)\xrightarrow{\sim} \Cl_K$), and $ \alpha_{Q, \mathfrak{a}_Q}\in K^\times$ is a complex number depending on the choices of  $Q$ and $\mathfrak{a}_Q$ (but not on $\pi$, $\Omega$ nor $f_k$).  
 \subsubsection{The case of old forms} %%%%%%%%%%%%%%%
We will now explain how to extend the identity (\ref{eq:MWprelim}) to the case of old forms. Let $d,N'$ be positive integers such that $dN'| N$ and consider a newform (i.e. new at finite places) $f_k\in \mathcal{B}^*_k(N')$ belonging to the automorphic representation $\pi$. Then we get an element $\nu_{d,N'}^* f_k\in \mathcal{B}_k(N)$ given by $z\mapsto f_k(dz)$. Recall the representatives $\mathfrak{p}_1,\ldots, \mathfrak{p}_h\in \mathcal{I}_K$ of the class group $\Cl_K$ defined in (\ref{eq:p_i}) and the associated Heegner forms $Q_i=[a,b_i,c_i]$ defined in (\ref{eq:Qi}). Then we see directly that 
 $$dz_{Q_i}= \frac{-b_i+\sqrt{D}}{2p_ia/d}=z_{Q_i'},$$  
 where $Q_i'= [p_ia/d, b_i, c_id]\in \mathcal{Q}_D(N',r)$ is a Heegner form of level $N'$ and orientation $r\modulo (2N')$. From this we see that  
 $$ f_k(dz_{Q_i})= \phi_{f_k}(\Psi'_\fin(\widehat{\mathfrak{p}_i})\gamma'_\infty), \quad i=1,\ldots ,h, $$
 where $\Psi'$ is the optimal embedding of level $N'$ corresponding to the triple $[a/d, b, cd]$ and
\begin{equation*} \gamma'_\infty=\begin{pmatrix} \sqrt{D} & -b \\ 0 & a/d\end{pmatrix}.\end{equation*}
Observe that $[a/d, b, cd]$ might not correspond to the trivial element of the class group.
Thus we conclude using (\ref{eq:MWprelim}) that 
\begin{align} 
\nonumber\sum_{[Q]\in \Gamma_0(N)\backslash \mathcal{Q}_D(N,r)} \nu_{d,N'}^* f_k(z_{Q})   \overline{\Omega(\widehat{\mathfrak{a}_Q})}\Omega_\infty(  \alpha_{Q, \mathfrak{a}_Q})&=\sum_{i=1}^h \nu_{d,N'}^* f_k(z_{Q_i})   \overline{\Omega(\widehat{\mathfrak{p}_i})} \\
\label{eq:MWoldforms}& =\int_{\A^\times K^\times \backslash \A_K^\times}\phi_{\mathrm{MW}}'(\Psi'_\A(x))\Omega^{-1}(x) dx , 
%&= \eps_{d,N'}\sum_{[Q']\in \Gamma_0(N')\backslash \mathcal{Q}_D(N',r)} f_k(z_{Q})   \overline{\Omega(\widehat{\mathfrak{a}_Q})}\Omega_\infty(  \alpha_{Q, \mathfrak{a}_Q})\\ 
 \end{align}
 where $\phi_{\mathrm{MW}}'$ is the vector defined by Martin and Whitehouse corresponding to the triple $(\pi, \Omega, \Psi'_\A)$ and the numbers $\alpha_{Q, \mathfrak{a}_Q}$ are as in (\ref{eq:MWprelim}).

%This implies that for $Q\in \mathcal{Q}_D(N,r)$ and $x_Q\in \A_K^\times$ as in Corollary \ref{cor:explicitformula}, we have 
%\begin{align} 
%\nu_{d,N'}^* f_k(z_Q)\Omega(x_Q)= \eps_{d,f_k} f_k(z_Q)\Omega(x_Q)    \\
% \frac{c_{f_k}}{|D|^{1/4}} \sum_{\chi\in \widehat{\Cl_K}} \eps_{\chi,f_k,r} |L(\pi\otimes \chi \Omega , 1/2)|^{1/2} \chi([\mathfrak{a}])\end{align}

%From this we get 
%$$  \int_{\A^\times K^\times \backslash \A_K^\times} \phi(x)\Omega^{-1}(x) dx
 %=\sum_{[\mathfrak{a}]\in \Cl_K} f_k(z_{Q(\mathfrak{a})}) \Omega_\infty(  \alpha_{Q(\mathfrak{a}), \mathfrak{a}})  \overline{\Omega(\hat{\mathfrak{a}})} \Omega(\widehat{\mathfrak{a}_{[\Psi]}}),    $$
% where $z_{Q(\mathfrak{a})}$ is the Heegner point associated to a Heegner form of level $N$ and orientation $r$ with $[Q(\mathfrak{a})]=[\mathfrak{a}]$ (under the isomorphism $\Gamma_0(N)\backslash \mathcal{Q}_D(N,r)\cong \Cl_K$), $\mathfrak{a}_{[\Psi]}\in\mathcal{I}_K$ is such that $[\mathfrak{a}_{[\Psi]}]=[\Psi]$ (under the isomorphism $\mathcal{E}_K(N,r)/\Gamma_0(N)\cong \Cl_K$) and $ \alpha_{Q(\mathfrak{a}), \mathfrak{a}}\in K^\times$ is a complex number depending the choices of $\mathfrak{a}$ and $Q(\mathfrak{a})$.  

Combining (\ref{eq:MWoldforms}) and (\ref{eq:WaldspurgerMW}) we arrive at the following result (recalling the definition (\ref{eq:basis}) of $\mathcal{B}_k(N)$).
\begin{thm} \label{thm:wald}Let $N$ be a square-free integer and $K$ an imaginary quadratic field of discriminant $D$ with $(D,2N)=1$ and such that all primes dividing $N$ splits in $K$. Let $\pi$ be a cuspidal automorphic representation of $\GL_2(\A_\Q)$ of conductor $N'$ dividing $N$ and even lowest weight $k_\pi$. Let $k\geq k_\pi$ be an even integer and $\Omega: K^\times\backslash \A_K^\times/ \widehat{\mathcal{O}_K}^\times\rightarrow \C^\times$ a Hecke character of $K$ of conductor $1$ and $\infty$-type $\alpha\mapsto (\alpha/|\alpha|)^k$. 

Then for any $f_k\in \mathcal{B}_k(N)$ belonging to the representation space of $\pi$, we have 
\begin{align}
\left|  \sum_{[Q]\in \Gamma_0(N)\backslash \mathcal{Q}_D(N,r)} f_k(z_{Q})   \overline{\Omega(\widehat{\mathfrak{a}_Q})}\Omega_\infty(  \alpha_{Q, \mathfrak{a}_Q}) \right|^2
= \frac{L (\pi\otimes \Omega,1/2)}{L(\sym^2 \pi, 1)} \frac{|D|^{1/2}}{8 N'} c_\infty(\pi_\infty, k)
\end{align}
where $z_{Q}$ is the Heegner point associated to the Heegner form $Q\in\mathcal{Q}_D(N,r)$, $\mathfrak{a}_Q\in \mathcal{I}_K$ is such that $[Q]=[\mathfrak{a}_Q]$ (under the bijection $\Gamma_0(N)\backslash \mathcal{Q}_D(N,r)\xrightarrow{\sim} \Cl_K$), $ \alpha_{Q, \mathfrak{a}_Q}\in K^\times$ is a complex number depending on the choices $Q$ and $\mathfrak{a}_Q$ (but not on $\pi$, $\Omega$ nor $f_k$),  and 
\begin{align} \label{eq:cinfty}c_\infty(\pi_\infty, k)=\begin{cases} (2\pi)^{k} \prod_{j=0}^{k/2-1}(1/4+(t_\pi)^2+j(j+1) )^{-1} & \text{if $\pi_\infty$ is a p.s,}\\
  (2\pi)^{k-k_\pi-1}\Gamma(k_\pi)/\left(\Gamma\left(\frac{k-2}{2}\right)B\left(\frac{k+k_\pi+1}{2}, \frac{k-k_\pi+1}{2}\right)\right)& \text{if $\pi_\infty$ is a d.s,}\end{cases}  \end{align}
where {\lq\lq}p.s{\rq\rq} and {\lq\lq}d.s{\rq\rq} refer to respectively, {\lq\lq}principal series{\rq\rq} and {\lq\lq}discrete series{\rq\rq}, and $B(x,y)$ denotes the Beta function.
\end{thm}

%Thus we conclude that 
%$$  \int_{\A^\times K^\times \backslash \A_K^\times} \phi(x)\Omega^{-1}(x) dx
% =\sum_{[\mathfrak{a}]\in \Cl_K} f_k(z_{Q(\mathfrak{a})}) \Omega_\infty(  \alpha_{Q(\mathfrak{a}), \mathfrak{a}, [\Psi]} )  \overline{\Omega(\hat{\mathfrak{a}})} \Omega(\widehat{\mathfrak{p}_{i([\Psi])}}),    $$
 %where $z_{Q(\mathfrak{a})}$ is the Heegner point associated to a Heegner form of level $N$ with $[Q(\mathfrak{a})]=[\mathfrak{a}]$ (under the isomorphism $\Gamma_0(N)\backslash \mathcal{Q}_D(N)\cong \Cl_K$), $i([\Psi])\in\{1,\ldots, h\}$ is such that $[\mathfrak{p}_{i([\Psi])}]=[\Psi]$ (under the isomorphism $\mathcal{E}_N/\Gamma_0(N)\cong \Cl_K$) and $ \alpha_{Q(\mathfrak{a}), \mathfrak{a}, [\Psi]}\in K^\times$ is a complex number depending on $[\Psi]$ and the choices of $\mathfrak{a}$ and $Q(\mathfrak{a})$.  
Using orthogonality of characters (i.e. Fourier inversion) we conclude the following key identity.
\begin{cor}\label{cor:explicitformula}
Let $\pi,\Omega, f_k$ be as in Theorem \ref{thm:wald}. Then given an element of the class group $[\mathfrak{a}]\in \Cl_K$ and a Heegner form $Q\in  \mathcal{Q}_D(N,r)$ such that $[Q]=[\mathfrak{a}]$, we have
\begin{equation}
f_k(z_Q) \Omega(x_{Q})=\frac{c_{f_k}|D|^{1/4}}{|\Cl_K|} \sum_{\chi\in \widehat{\Cl_K}} \eps_{\chi,f_k,r} |L(\pi\otimes \chi \Omega , 1/2)|^{1/2} \chi([\mathfrak{a}]),  
\end{equation}
where $x_{Q}\in \A_K^\times$ is some element depending on the choice of $Q$ (but not on $\pi$, $\Omega$ nor $f_k$), $\eps_{\chi,f_k,r}$ are complex numbers of norm 1, and %$c_{f_k}$ is some constant depending only on $f_k$ (maybe write explicit bound in terms of $t_\phi$ and $k$). 
\begin{equation}\label{eq:ctestvector}c_{f_k}=\frac{c_\infty(\pi_\infty, k)}{8 N' L(\sym^2 \pi, 1)}, \end{equation}
with $c_\infty(\pi_\infty, k)$ as in (\ref{eq:cinfty}).%ic_f_k depends on the Petersson norm of f_k!  
 \end{cor}
 
%Let $N$ be a square-free integer and $K$ be an imaginary quadratic field of discriminant $D$ with $(D,2N)=1$ and such that all primes dividing $N$ splits in $K$. Let $f_k\in \mathcal{B}_k(N)$ belong to the cuspidal automorphic representation $\pi$ of level $N'$. Fix a Hecke character of conductor $1$ and $\infty$-component $\alpha\mapsto (\alpha/|\alpha|)^k$ with $k\in 2\Z$ and let $\Omega:K^\times \backslash \A_K^\times/\mathcal{O}_K^\times\rightarrow \C^\times$ be the corresponding id\'elic character.  Then we have 
 
 \section{Some technical lemmas}\label{sec:technical}
In this section we will prove two key estimates. The first is a bound for the norm of $\Delta^m$, which will be key in obtaining explicit error-terms in our moment calculation. Similar consideration have been made in a different context in \cite[Theorem 5.1]{PetRis18}. Secondly, we will obtain a lower bound for the $L^2$-norm of the product of Maa{\ss} forms. This is an extremely crude lower bound, which however suffices for our purposes.
 \subsection{A bound for the norm of $\Delta^m$}
In the course of proving our bound for the norm of $\Delta^m$ applied to certain vectors, we will need the following convenient $L^\infty$-bound for $f\in \mathcal{B}_k(N)$ due to Blomer and Holowinsky \cite{BlomerHolo10}:
\begin{equation}\label{eq:BH}||f ||_\infty/||f ||_2\ll N^{-1/32}(|t_f|+|k|+1)^A,\end{equation}
%$$||R_{k+2l}R_{k+2l-2}\cdots R_kf ||_\infty/||R_{k+2l}R_{k+2l-2}\cdots R_kf ||_2\ll N^{-1/32}(|t_f|+l+1)^A,$$
for some unspecified constant $A>0$. The focus of \cite{BlomerHolo10} is the level aspect, which we consider fixed in the present paper. Here the key thing is however that we get a polynomial bound for raised (and lowered) Hecke--Maa{\ss} forms with the constant being independent of the weight $k$ and the spectral parameter $t_f$. The specific value of $A$ is not important for our application. %We will now proceed to prove the following.
\begin{lemma}\label{lem:supDelta}
Let $k_1,\ldots, k_n$ be even integers such that $\sum_{i=1}^n k_i=0$. For $i=1,\ldots, n$, let $f_i\in \mathcal{B}_{k_i}(N)$ be a Hecke--Maa{\ss} form of weight $k_i$, level $N$ and spectral parameter $t_{f_i}$. Then we have
\begin{equation} ||\Delta^m \prod_{i=1}^n f_i ||_\infty \ll n^{2m} (m+\max_{i=1,\ldots, n} |t_{f_i}|+|k_i|)^{nA+2m} \prod_{i=1}^n ||f_i||_2,\end{equation}
for all $m\in \Z_{\geq 0}$. Here the implied constant is allowed to depend on $N$. 
\end{lemma}
\begin{proof}
Recalling that $\Delta =L_2 R_0$, we get using the product rule for the raising and lowering operators:
\begin{align}
\nonumber |\Delta^m \prod_{i=1}^n f_i(z)| &= | L_2 R_0\cdots L_2 R_0 \prod_{i=1}^n f_i(z)|\\
\label{eq:maxestimate}& \leq n^{2m} \max_{\substack{m_1,\ldots, m_n \in \N, \Sigma\, m_i=2m, \\U_{i,j}\text{ evenly raising/lowering operators}\\\text{for } i=1,\ldots,n \text{ and }  j=1,\ldots, m_i}} \prod_{i=1}^n |U_{i,1}\cdots U_{i,m_i}\, f_i(z)|,
\end{align}
where the maximum is taken over all combinations of operators $U_{i,j}$ which are all either a raising or a lowering operator of appropriate weight and such that the total number of raising and lowering operators are equal. If we have $i\in \{1,\ldots, n\}$ and $j\in \{1,\ldots, m_i-1\}$ such that $\{U_{i,j},U_{i,j+1}\}$ is of the type $\{$raising, lowering$\}$, then we get
$$  U_{i,j}U_{i,j+1}=-\Delta_{\pm \kappa} +\lambda(\kappa/2), $$
for some weight $\kappa$ with $|\kappa|\leq 2m+|k_i|$ (since we can have at most $m$ raising resp. lowering operators). Here the sign corresponds to whether $U_{i,j}$ is a raising or lowering operator. This shows that we can replace $U_{i,j}U_{i,j+1}$ with multiplication by 
$$ \lambda( \kappa/2)-\lambda_{f_i}=-( (\kappa-1)/2+i t_{f_i} )((\kappa-1)/2-i t_{f_i} ).$$ 
Repeating this we get
\begin{align*}
&|U_{i,1}\cdots U_{i,m_i}\, f_i(z)| \\
&= \left|R_{k+2m_i'-2}\cdots R_k\, f_i(z) \prod_{j=1}^{(m_i-m_i')/2} ( (\kappa_j-1)/2+i t_f )((\kappa_j-1)/2-i t_f )\right|.
\end{align*}
for some $0\leq m_i'\leq m_i$, where $|\kappa_j|\leq 2m+|k_i|$ (or a similar expression with lowering instead of raising operators). 

By combining the bound (\ref{eq:BH}) and the computation of the $L^2$-norm (\ref{eq:normraising}), we conclude that for $f\in\mathcal{B}_k(N)$ and $l\geq 0$:
\begin{align*}  
&||R_{k+2l}R_{k+2l-2}\cdots R_kf ||_\infty \\
&\ll ||f||_2  (|t_f|+|k+l|+1)^A  \prod_{j=0}^l \left|\left(\frac{k+2j-1}{2}+it_f\right)\left(\frac{k+2j-1}{2}-it_f\right)\right|^{1/2} \\
&\ll  ||f||_2 (|t_f|+|k|+l+1)^{l+A},
\end{align*}
and similarly in the case of lowering operators. Combining all of the above, we arrive at 
$$ |U_{i,1}\cdots U_{i,m_i}\, f_i(z)| \ll  ||f_i||_2  (|t_f|+|k_i|+m_i+1)^{A+m_i},  $$
for any sequence of raising and lowering operators $U_{i,1},\ldots, U_{i,m_i}$ as in the maximum in (\ref{eq:maxestimate}). Plugging this into (\ref{eq:maxestimate}) gives the wanted.
\end{proof}
\subsection{A lower bound for weight $k$ automorphic forms}%%%%%%%%%%%%%%%%%%%
In this subsection, we will prove a lower bound for the $L^2$-norm of a product of Maa{\ss} forms. The idea is to go far up in the cusp so that the first term in the Fourier expansion is the dominating term. 

%\subsubsection{The Maa{\ss} case}
Let $W_{k/2,s}:\R_{>0}\rightarrow \C$ be the {\it Whittaker function} of weight $k/2$ and spectral parameter $s$, i.e. the unique solution to
$$  \frac{d^2W}{dy^2}+ \left(-\frac{1}{4}+\frac{k/2}{y}+\frac{1/4-s^2}{y^2}\right)W=0, $$
satisfying
$$ W_{k/2,s}(y)\sim y^{k/2}e^{-y/2},  $$
as $y\rightarrow \infty$ (with $k,s$ fixed).   Then we define $\mathcal{W}_{k/2,s}: \C\backslash \R \rightarrow \C$  for $k\in \Z$ as 
$$\mathcal{W}_{k/2,s}(z):=\begin{cases} (-1)^{k/2}W_{|k|/2,s}(|y|)e^{ix/2},& \sgn (k)y>0,\\ \frac{\Gamma\left(\frac{|k|+1}{2}+s\right)\Gamma\left(\frac{|k|+1}{2}-s\right)}{\Gamma\left(\frac{1}{2}+s\right)\Gamma\left(\frac{1}{2}-s\right)}W_{-|k|/2,s}(|y|)e^{ix/2},& \sgn(k) y<0,\end{cases}$$ 
for $z=x+iy\in  \C\backslash \R$. One can check that
$$\mathcal{W}_{0,s}(z)=\left(\frac{|y|}{\pi}\right)^{1/2} K_{s}(|y|/2)e^{ ix/2},$$ 
where $K_s(y)$ is the $K$-Bessel function and 
$$ \mathcal{W}_{k/2,(k-1)/2}(z)=y^{k/2}e^{iz/2}, $$
for $k\in 2\Z_{\geq 0}$ and $y>0$. Furthermore, for $k\in 2\Z_{\geq 0}$ one can check (see for instance \cite[Section 4.4]{Stromberg08}) that the normalizations match up so that we have
\begin{equation}\label{eq:raisingwhit}R_k\mathcal{W}_{k/2,s}=\mathcal{W}_{k/2+1,s},\end{equation}
with 
$$R_k=(z-\overline{z})\frac{\partial}{\partial z}+\frac{k}{2}=iy\frac{\partial }{\partial x}+y\frac{\partial }{\partial y}+\frac{k}{2},$$ 
denoting the weight $k$ raising operator (and similarly for $k\leq 0$ now with lowering operators). We have the following asymptotic expansion (see \cite[9.227]{GradRyz00} or \cite[Chapter 16.3]{WhittakerWatson62}) valid for $y>1$:
\begin{align} 
\nonumber W_{k/2,s}(y)&= e^{-y/2} y^{k/2} \left(1+\sum_{n\geq 1}\frac{(s^2-(k/2-1/2)^2)\cdots (s^2-(k/2-n+1/2)^2)}{n! y^n}\right).\end{align}
In particular we conclude that 
\begin{align}
\nonumber W_{k/2,s}(z)&= e^{-y/2} y^{k/2} \left( 1+ O\left(\sum_{n\geq 1}\frac{(|s|+|k|/2+n)^{2n}}{n! y^n} \right)\right)\\
\label{eq:asymptWhit}&= e^{-y/2} y^{k/2} \left( 1+ O\left( \frac{(|s|+|k|+1)^2}{y}\right)\right),
\end{align}
for $y>(|s|+|k|+1)^2$.

Now let $k\geq 0$ and consider an $L^2$-normalized Hecke--Maa{\ss} form $f \in \mathcal{B}_k(N)$ of the form $\nu^*_{d,N'}R_{k-2}\cdots R_{k'} f_0$ with $f_0$ a Hecke--Maa{\ss} newform of weight $k'$ and level $N'$ such that $dN'|N$. Combining (\ref{eq:raisingwhit}) and (\ref{eq:normraising}) with the well-known Fourier expansions of holomorphic and Maa{\ss} forms, we get the following Fourier expansion in the general weight case:
\begin{equation}\label{eq:Fourerphi}f(z)=\frac{c_{f}}{ |L(\sym^2 f, 1) \gamma_\infty(f,k)|^{1/2}}\sum_{n\neq 0} \frac{\lambda_{f_0}(n)}{|n|^{1/2}}\mathcal{W}_{k/2,it_f}(4\pi dnz),  \end{equation}
for some constant $c_{f}$ bounded uniformly from above and away from $0$ in terms of the level $N$. Here $\lambda_{f_0}(n)$ denotes the Hecke eigenvalues of $f$ (with the convention that $\lambda_{f_0}(-n)=0$ for $-n<0$ if $f_0$ is holomorphic and $\lambda_{f_0}(-n)=\pm \lambda_{f_0}(n)$ according to whether $f_0$ is an even or odd Maa{\ss} form) and
$$  \gamma_\infty(f,k)=\begin{cases} \prod_\pm \Gamma((k+1)/2\pm it_f)& \text{if $f_0$ is Maa{\ss}},\\
 \Gamma(k)\Gamma((k-k')/2+1) &  \text{if $f_0$ is holomorphic.}\end{cases}$$
Using this we can prove the following crude lower bound.
\begin{prop}\label{prop:lowerboundL2General}
For $i=1,\ldots, n$, let $f_i\in \mathcal{B}_{k_i}(N)$ be an $L^2$-normalized weight $k_i$ Hecke--Maa{\ss} eigenform of level $N$. Then we have
$$ ||\prod_{i=1}^n f_i||_2 \gg_\eps e^{-cnT^{2+\eps}}, $$
for all $\eps>0$, where $T=\max_{i=1,\ldots, n} |t_{f_i}|+|k_i|+1$ and $c=c(N,\eps)>0$ is some positive constant.
\end{prop}
\begin{proof}
Clearly we may assume that $k\geq 0$. Given $f\in \mathcal{B}_k(N)$ we write 
$$f=\nu^*_{d,N'}R_{k-2}\cdots R_{k'}f_0,$$ 
for a Hecke--Maa{\ss} newform $f_0$ of weight $k'$ (with $k'\leq k$ and $k'\equiv k\modulo 2$) and level $N'$ with $dN'|N$. We have by a standard bound for the Hecke eigenvalues (see for instance \cite[(8.7)]{Iw} in the Maa{\ss} case) and by bounding the quotient of $\Gamma$-factors trivially:
\begin{align}
\label{eq:approxWhitt}  &\sum_{n\neq 0} \frac{\lambda_{f}(n)}{|n|^{1/2}}\mathcal{W}_{k/2,it_f}(4\pi nz) \\
\nonumber& = e^{2\pi i d x}W_{k/2,it_f}(4\pi dy)+\eps_f e^{-2\pi i d x}  \frac{\Gamma\left(\frac{k+1}{2}+s\right)\Gamma\left(\frac{k+1}{2}-s\right)}{\Gamma\left(\frac{1}{2}+s\right)\Gamma\left(\frac{1}{2}-s\right)} W_{-k/2,it_f}(4\pi dy)\\
\nonumber &+O\left( |t_f|^{1/2}\sum_{n\geq 2} |W_{k/2,it_f}(4\pi dny)|+(k+|t_f|+1)^k|W_{-k/2,it_f}(4\pi dny)| \right),
\end{align}
where $\eps_f=0$ if $f_0$ is holomorphic and if $f_0$ is a Maa{\ss} form we have $\eps_f=\pm 1$ where $\pm 1$ is the sign of $f_0$ under the reflection operator $X$ defined in Section \ref{sec:autoforms}. By the asymptotics (\ref{eq:asymptWhit}) one sees easily that;
$$\sum_{n\geq 2} |W_{k/2,it_f}(4\pi dny)| +(k+|t_f|+1)^k|W_{-k/2,it_f}(4\pi dny)|\ll e^{-3d\pi y},$$
for $y\geq (|t_f|+k+1)^{2+\eps}$. For $k=0$ we conclude from the asymptotic (\ref{eq:asymptWhit}) that (\ref{eq:approxWhitt}) is equal to
$$  (e^{2\pi i d x}+\eps_f e^{-2\pi i d x})e^{-2\pi dy}+O(y^{-\eps}e^{-2\pi dy}), $$
for $y\geq (|t_f|+k+1)^{2+\eps}$. Similarly $k>0$ we see that (\ref{eq:approxWhitt}) is equal to
$$ e^{2\pi i d x} (4\pi dy)^{k/2}e^{-2\pi dy} +O((4\pi dy)^{k/2-\eps}e^{-2\pi dy}) $$
for $y\geq (|t_f|+k+1)^{2+\eps}$, using the bound
\begin{align*}\frac{\Gamma\left(\frac{k+1}{2}+s\right)\Gamma\left(\frac{k+1}{2}-s\right)}{\Gamma\left(\frac{1}{2}+s\right)\Gamma\left(\frac{1}{2}-s\right)} W_{-k/2,it_f}(4\pi dy) &\ll (k+|t_f|+1)^k (4\pi dy)^{-k/2} e^{-2\pi dy}. \end{align*}

By Stirling's approximation we have the crude bound 
$$\gamma_\infty(f,k)\ll e^{O((|t_f|+k)\log (|t_f|+k))}, $$ 
and we also have  $|t_f|^{-\eps}\ll_\eps L(\sym^2 f, 1) \ll_\eps |t_f|^\eps $. Thus we conclude from (\ref{eq:Fourerphi}) that for $k=0$, say:
\begin{equation}\label{eq:lowerboundphi} |f(z)|\gg  e^{-3\pi  dy },  \end{equation}
for $y\geq (|t_f|+k+1)^{2+\eps}$ and $x$ such that $e^{2\pi i d x}+\eps_f e^{-2\pi i d x}\gg 1$. Similarly if $k>0$, we have, say: 
\begin{equation}\label{eq:lowerboundphi2} |f(z)|\gg e^{-3\pi  dy },  \end{equation}
for $y\geq (|t_f|+k+1)^{2+\eps}$ (and any $x$). Now we easily conclude the wanted lower bound for the $L^2$-norm of the product by computing the contribution from the range $x\in[0,1]$ and $y\asymp (|t_f|+k+1)^{2+\eps}$.
\end{proof}

In the holomorphic case, we can do slightly better since the Fourier expansion is better behaved.
\begin{prop}\label{prop:lowerboundL2holo}
For $i=1,\ldots, n$, let $f_i\in \mathcal{B}_{k_i,\mathrm{hol}}(N)$ be a weight $k_i$ holomorphic Hecke--Maa{\ss} eigenform of level $N$ ($L^2$-normalized). Then we have
$$ ||\prod_{i=1}^n f_i||_2 \gg_\eps e^{-cnT^{1+\eps}}, $$
for all $\eps>0$, where $T=\max_{i=1,\ldots,n } |k_i|$ and $c(N,\eps)=c>0$ is some positive constant.
\end{prop}
\begin{proof}
Let $f\in \mathcal{B}_{k,\mathrm{hol}}(N)$ be of the form $\nu^*_{d,N'}y^{k/2}g$ with $g\in \mathcal{S}_k(N')$ a holomorphic Hecke newform. By the Fourier expansion (\ref{eq:Fourerphi}),  we have
$$f(z)=\frac{c_{f}}{ |L(\sym^2 f, 1) \Gamma(k)|^{1/2}}\sum_{n\geq 1} \frac{\lambda_g(n)}{n^{1/2}}(4\pi d n y)^{k/2} e^{2\pi i dnz}.$$
By bounding everything trivially, it is easy to see that for $y\gg k^{1+\eps}$;
$$\sum_{n\geq 1} \frac{\lambda_g(n)}{n^{1/2}}(4\pi dn y)^{k/2} e^{2\pi i dnz}=(4\pi d y)^{k/2} e^{2\pi i d z}+O_\eps(e^{-3\pi d y}). $$
Now the lower bound for $||\prod_{i=1}^n f_i||_2$ follows as above.
\end{proof}

\begin{remark}
It seems quite hard to obtain strong lower bounds for $||\prod_i f_i||_2$ as this is related to the deep problem of non-localization of the eigenfunctions $f_i$ (such as $L^\infty$-bounds), see for instance \cite{Sa95}. In particular it is very hard to rule out that the $f_i$'s localize in disjoint regions. \end{remark}

\section{Proof of main theorem}\label{sec:main}
We will now use the results proved in the previous sections to obtain our wide moment calculation. First of all we will use the above to obtain a version of equidistribution of Heegner points with explicit error-terms. For this we will need the following convenient basis for the space spanned by Maa{\ss} forms of square-free level $N$ (see \cite[Lemma 3.1]{HumKhan20}):
\[\mathcal{B}'(N):=\left\{u_{d} \in C^\infty(\Hb)\cap L^2(\Gamma_0(N)\backslash \Hb) :    N'd|N, u \in \mathcal{B}^{\ast}(N') \right\},\]
(recall that we denote by $\mathcal{B}^{\ast}(N')$ all Hecke--Maa\ss{} newforms $f$ of weight $0$ and level $N'$) where
\begin{equation}\label{eq:oldform}u_{d}(z) := \left(L_d(\sym^2 u,1) \frac{\varphi(d)}{d \nu(N/N')}\right)^{\frac{1}{2}} \sum_{vw = d} \frac{\nu(v)}{v} \frac{\mu(w) \lambda_u(w)}{\sqrt{w}} u(vz).\end{equation}
Here
\[L_{d}(\sym^2 u,s) := \prod_{p | d} \frac{1}{1 - \lambda_u(p^2) p^{-s} + \lambda_u(p^2) p^{-2s} - p^{-3s}}.\]
There is a similar basis for the Eisenstein part of the spectrum (see \cite[Section 3.2]{HumKhan20}). Given $u\in \mathcal{B}'(N)$ we put
$$ L(\sym^2 u,s):=L(\sym^2 u',s),\quad \text{and}\quad L(u,s):=L(u',s),  $$
where $u=(u')_d$ with $u'\in \mathcal{B}^{\ast}(N')$ and $dN'|N$.
%Idea to notation:
%$$ \mathbf{Wide}(n,\Cl_K)=\{(\chi_1,\ldots, \chi_n\in (\widehat{\Cl_K})^n:  \chi_1\cdots \chi_n=1\}$$
%$$ \mathbf{Wide}^*(2n,K)=\{(\chi_1,\ldots, \chi_n,\psi_1,\ldots, \psi_n)\in (\widehat{\Cl_K})^{2n}:  \chi_1\cdots \chi_n=\psi_1\cdots \psi_n\}.$$
\begin{thm}\label{thm:maincomp}
Let $k_1,\ldots, k_n\in 2\Z$ be even integers such that $\sum k_i=0$. For $i=1,\ldots, n$, let $f_i\in \mathcal{B}_{k_i}(N)$ be a Hecke--Maa{\ss} eigenform of fixed level $N$, weight $k_i$ and spectral parameter $t_{f_i}$. Let $|D_K|\rightarrow \infty$ transverse a sequence of discriminants of imaginary quadratic fields $K$ such that all primes dividing $N$ splits in $K$. Then we have 
%(BETTER TO WRITE IN TERMS OF GALOIS ORBITS?) NOTE $z_{[a],N}$ DEPENDS ON THE CHOICE OF ORIENTATION $r\modulo 2N$)!!
\begin{align*}
%&\sum_{\substack{\chi_1,\ldots, \chi_n\in \widehat{\Cl_{K}}\\ \prod \chi_i=1}} \prod_{i=1}^n \eps_i (\text{some positive factor $|D_K|^{-1/4}$...}) L(f_i\otimes \chi_i \lambda_{i,K},1/2)^{1/2}\\
&\frac{1}{|\Cl_K|}\sum_{[Q]\in \Gamma_0(N)\backslash \mathcal{Q}_{D_K}(N,r)} \,\,\prod_{i=1}^n f_i(z_{Q})\\
&= \langle \prod_{i=1}^n f_i,\frac{1}{\vol (\Gamma_0(N))}\rangle +O_\eps \left( || \prod_{i=1}^n f_i||_2 |D_K|^{-1/16 }T^5n^{5} (T|D_K|n)^\eps\right),
\end{align*}  
where $T= \max_{i=1,\ldots, n} |t_{f_i}|+|k_i|+1$. %The implied constant is allowed to depend on the level $N$. 

%We can improve the exponent to: $|D_K|^{-1/12 }T n^{3/2}$ if the level is $N=1$, $|D_K|^{-1/16 }T^{5/4}n^{15/4}$ if all $f_i$ are holomorphic newforms (of general level), and $|D_K|^{-1/12 }T^{1/2}n^{3/2}$ if all $f_i$ are holomorphic newforms of level $1$.

We have the following improvements for the exponents in the error-term: 
\begin{equation}\label{eq:caseserrorterm}  \begin{cases}  |D_K|^{-1/16 }T^{5/2}n^{5},& \text{if all $f_i$ are holomorphic},\\ |D_K|^{-1/12 }T^2 n^{2},& \text{if the level is $N=1$},\\|D_K|^{-1/12 }Tn^{2},& \text{if all $f_i$ are holomorphic of level $1$}. \end{cases}\end{equation} 
%If the level is $N=1$ we can improve the error term to
%$$ O_\eps \left( || \prod_{i=1}^n f_i||_2 |D_K|^{-1/12 }T n^{3/2} (T|D_K|n)^\eps \right), $$
%and if furthermore all $f_i$ are holomorphic with $t_{f_i}=(k_i-1)/2$ then we get the further improvement
%$$ O_\eps \left( || \prod_{i=1}^n f_i||_2 |D_K|^{-1/12 }T^{1/2}n^{3/2} (k|D_K|)^\eps \right). $$
  %N squarefree here. We know that $\prod_i \lambda_i$ is a class group character.
\end{thm}

\begin{proof}
We put $D=|D_K|$ to lighten notation. By the spectral expansion for $\Gamma_0(N)\backslash \Hb$ (see \cite[Theorem 7.3]{Iw}), we have
\begin{align}
\nonumber &\sum_{[Q]\in \Gamma_0(N)\backslash \mathcal{Q}_{-D}(N,r)} \prod_{i=1}^n f_i(z_Q)\\
\label{eq:experror}&= |\Cl_K| \langle\prod_{i=1}^n f_i,1 \rangle+\sum_{u\in \mathcal{B}'(N)} \langle\prod_{i=1}^n f_i, u\rangle W_{u,K} +(\mathrm{Eisenstein}),
\end{align}
where 
$$W_{u,K}:=\sum_{[Q]\in \Gamma_0(N)\backslash \mathcal{Q}_{-D}(N,r)} u(z_Q),$$ 
is the Weyl sum of level $N$ corresponding to $u$. By Theorem \ref{thm:wald}, we have
\begin{equation}\label{eq:boundwald} |W_{u,K}|^2\ll_N \frac{D^{1/2}L(u,1/2)L(u\otimes \chi_{K},1/2)}{L(\sym^2 u,1)}, \end{equation}
for $u\in \mathcal{B}'(N)$. Here the case when $u$ is a linear combination of oldforms as in (\ref{eq:oldform}) follows by linearity. Now we observe that for $u\in \mathcal{B}^*(N)$, we have using the self adjointness of $\Delta$:
\begin{align*}\langle \prod_{i=1}^n f_i, u\rangle (t_u^2+1/4)^m =  \langle\prod_{i=1}^n f_i, \Delta^m u\rangle =  \langle \Delta^m \prod_{i=1}^n f_i, u\rangle, \end{align*}
Applying Cauchy--Schwarz and Lemma \ref{lem:supDelta}, this implies
\begin{align}\label{eq:boundm}\langle \prod_{i=1}^n f_i, u\rangle &\ll \prod_{i=1}^n ||f_i||_2  \frac{n^{2m} (m+T)^{nA+2m} }{(|t_u|^2+1)^m}, \end{align}

for any $m\geq 0$, where $T=\max_{i=1,\ldots, n} |t_{f_i}|+|k_i|+1$. Putting $m=(nT^2)^{1+\eps}$ in the estimate (\ref{eq:boundm}), we see that one can truncate the spectral expansion (\ref{eq:experror}) at $t_u \ll (Tn)^2 (TDn)^\eps$ at the cost of an error of size 
\begin{equation*} \ll_\eps  (TDn)^{-c(nT^2)^{1+\eps}} \prod_{i=1}^n ||f_i||_2  ,\end{equation*}
for some constant $c=c(N,\eps)>0$. By Proposition \ref{prop:lowerboundL2General}, this error is negligible.  %HERE THE TWO EPSILONS ARE DIFFERENT! THE ONE IN m SHOULD BE DIVIDED BY, SAY, 4.

To estimate the remaining terms, we use the bound (\ref{eq:boundwald}) together with Cauchy--Schwarz and Bessel's inequality, non-negativity and standard bounds for symmetric square $L$-functions. This gives;
\begin{align}
\label{eq:uprodphi}&\sum_{\substack{u\in \mathcal{B}'(N),\\ t_u\ll (Tn)^{2} (TDn)^\eps}} \langle\prod_{i=1}^n f_i, u\rangle W_{u,K}\\
\nonumber &\ll_\eps ||\prod_{i=1}^n f_i||_2 D^{1/4} \left(\sum_{N'|N} \sum_{\substack{u\in \mathcal{B}^*(N'),\\ t_u\ll (Tn)^{2} (TDn)^\eps}} L(u,1/2)  L(u\otimes \chi_{K},1/2) \right)^{1/2}  (TDn)^\eps,
\end{align}
where $\chi_K$ is the quadratic character corresponding to $K$ via class field theory (recall that $\mathcal{B}^*(N')$ denotes the set of all Hecke--Maa{\ss} newforms of weight $0$ and level $N'$).

From here on we distinguish between the case of level 1 and higher (square free) level $N$. In the case of general level $N$, we use the $\GL_2$ sub-convexity bound due to Blomer and Harcos \cite{BlomerHarcos08};
$$ L(u\otimes \chi_{K},1/2) \ll (1+|t_u|)^{3+\eps} D^{3/8+\eps}, $$
which gives 
\begin{align*}
&\sum_{\substack{ u\in \mathcal{B}'(N)\\ t_u\ll (T n)^{2} (TDn)^\eps}} \langle\prod_{i=1}^n f_i, u\rangle W_{u,K}\\
&\ll ||\prod_{i=1}^n f_i||_2 D^{1/4+3/16} (Tn)^{3} (TDn)^\eps \left(\sum_{N'|N}\sum_{\substack{u\in \mathcal{B}^*(N') \\ t_u\ll (T n)^{2} (TDn)^\eps}} L(u,1/2)  \right)^{1/2}\\
&\ll     || \prod_{i=1}^n f_i||_2 D^{1/2-1/16} (Tn)^{5} (TDn)^\eps,
\end{align*}
using a standard first-moment bound for $L(u,1/2)$ (for instance using a spectral large sieve). As usual, the Eisenstein contribution can be bounded similarly.

If the level is $1$, we follow Young \cite{Young17} and use H\"older's inequality together with his Lindel\"of strength third moment bound  \cite[Theorem 1.1]{Young17} to estimate the above by
$$ \ll_\eps ||\prod_{i=1}^n f_i||_2 D^{5/12} (Tn)^{2}(TDn)^\eps. $$ % which is standard using a spectral large sieve inequality (CHECK THIS)

Finally, if all of the $f_i$ are holomorphic, then by Proposition \ref{prop:lowerboundL2holo} we can use the estimate (\ref{eq:boundm}) with $m=nT^{1+\eps}$ instead, which leads to the improved exponents. %(which is known or not? for prime level it follows from third moment bound of Young, but should be easy). 
\end{proof}

\begin{remark}\label{remark:ET}
Alternatively, one can estimate (\ref{eq:uprodphi}) by using the bound:   
$$ \langle \prod_i f_i, u\rangle \ll_\eps t_u^{5/12+\eps}||\prod_i f_i||_1,  $$
where $||\cdot||_1$ denotes the $L^1$-norm, using here the $L^\infty$-bound of Iwaniec and Sarnak \cite{IwSa}. This leads to the following error-term
$$ O_\eps \left( || \prod_{i=1}^n f_i||_1 |D_K|^{-1/16}T^{35/6}n^{35/6} (T|D_K|n)^\eps\right) $$
which is more convenient in some cases (with similar improvements in the special cases of holomorphic and/or level $1$ as in (\ref{eq:caseserrorterm})).
\end{remark}

\subsection{A wide moment of $L$-functions}
Combining this with our explicit formula, we arrive at our main $L$-function computation. We will use the following shorthand for $K$ an imaginary quadratic field with class group $\Cl_K$;
$$\mathbf{Wide}(K,n):=\mathbf{Wide}(\widehat{\Cl_K},n),$$
with $\mathbf{Wide}(G,n)$ as in (\ref{eq:Wide}). Note that the following statement is a slight generalization of Theorem \ref{thm:main} (allowing for the representations not to have the same conductor).
\begin{cor}\label{cor:main}
%For $i=1,\ldots, n$, let $f_i\in A(\Gamma_0(N)\backslash \H,  k_i)$ be a weight $k_i\in 2 \Z$ Hecke new at all finite places such that $\sum \kappa_i=0$.  
Let $N\geq 1$ be a fixed square-free integer. For $i=1,\ldots, n$, let $\pi_i$ be a cuspidal automorphic representation of $\GL_2(\A)$ with trivial central character of conductor $N_i|N$, spectral parameter $t_{\pi_i}$ and even lowest weight $k_{\pi_i}$. Let $k_1,\ldots, k_n\in 2\Z$ be integers such that  $|k_i|\geq k_{\pi_i}$ and $\sum_i k_i=0$. 

Let $|D_K|\rightarrow \infty$ transverse a sequence of discriminants of imaginary quadratic fields $K$ such that all primes dividing $N$ splits in $K$. For each $K$ pick Hecke characters $\Omega_{i,K}$ with $\infty$-type $x\mapsto (x/|x|)^{k_i}$ such that $\prod_i \Omega_{i,K}$ is the trivial Hecke character (notice that this is always possible since, we know that $\prod_i \Omega_{i,K}$ is a class group character).

Then we have for $f_i\in \mathcal{B}_{k_i}(N)$ in the representation space of $\pi_i$;
\begin{align}
\nonumber&\sum_{(\chi_i)\in \mathbf{Wide}(K,n)} \prod_{i=1}^n  c_{f_i} \eps_{\chi_i,f_i} L(\pi_i\otimes \chi_i \Omega_{i,K},1/2)^{1/2}\\
\label{eq:maincomp}&=  \frac{|\Cl_K|^n}{|D_K|^{n/4}}\Biggr(\langle\prod_{i=1}^n f_i,\frac{1}{\vol(\Gamma_0(N))}\rangle+O_\eps \left( || \prod_{i=1}^n f_i||_2 |D_K|^{-1/16}T^5n^{5} (T|D_K|n)^\eps\right)\Biggr),
\end{align}  
where $T= \max_{i=1,\ldots, n} |k_i|+|t_{f_i}|+1$, $c_{f_i}=(8N_i)^{-1}c_\infty(\pi_{i,\infty},k_i)$ with $c_\infty$ as in (\ref{eq:cinfty}) and $\eps_{\chi,f_i}$ complex numbers of norm $1$. 
%where $f_i=\otimes f_{i,p}\in \pi_i$ are test vectors new at all finite places and of weight $k_i$ at $\infty$ and $L^2$-normalized with respect to the Petersson norm. The constants $c_{\pi_i,k_i}$ are given by ?? and $\eps_{\chi,i}$ are all of norm $1$. 
%The implied constant is allowed to depend on the level $N$.  

We have the following improvements for the exponents in the error-term: 
\begin{equation} \label{eq:errortermscormain} \begin{cases}  |D_K|^{-1/16 }T^{5/2}n^{5},& \text{if $\pi_i$ are discrete series of weight $k_{\pi_i}=k_i$},\\ |D_K|^{-1/12 }T^2 n^{2},& \text{if the level $N=1$ is trivial},\\|D_K|^{-1/12 }Tn^{2},& \text{if $N=1$ and $\pi_i$ are discrete series of weight $k_{\pi_i}=k_i$}. \end{cases}\end{equation} 

%If the level is $N=1$ we can improve the error term to
%\begin{equation}\label{eq:errorL1}O_\eps \left( || \prod_{i=1}^n f_i||_2 |D_K|^{-1/12 }Tn^{3/2} (T|D_K|n)^\eps \right), \end{equation}
%and if furthermore all $\pi_i$ are discrete series representations then we get the further improvement
%$$ O_\eps \left( || \prod_{i=1}^n f_i||_2 |D_K|^{-1/12 }T^{1/2}n^{3/2} (T|D_K|n)^\eps \right).$$

%as $D\rightarrow \infty$ with $D$ a fundamental discriminant such that all primes dividing $N$ splits in $\Q[\sqrt{-D}]$, where $|\eps_i|=1$ for $i=1,\ldots, n$.  Here $\lambda_{i,-D}$ is a Hecke character of $\Q[\sqrt{-D}]$ with infinite type $x\mapsto (x/|x|)^{k_i}$ chosen so that $\prod_i \lambda_{i,-D}=1\in \widehat{\Cl_{-D}}$. %N squarefree here. We know that $\prod_i \lambda_i$ is a class group character.
\end{cor}
\begin{proof}
By the fact that $\prod_i \Omega_{i,K}$ is trivial, we see that
$$ \prod_{i=1}^n f_i(z)=\prod_{i=1}^n \Omega_{i,K}(x)f_i(z)  $$
for any $x\in \A_K^\times$. In particular, if we fix a quadratic form $Q\in \mathcal{Q}_{D_K}(N,r)$ and choose $x=x_Q\in \A_K^\times$ as in Corollary \ref{cor:explicitformula}, then we get
\begin{align*}
\prod_{i=1}^n f_i(z_Q)&=\prod_{i=1}^n \Omega_{i,K}(x_Q)f_i(z_{Q})\\
& =\sum_{\chi_1,\ldots, \chi_n\in \widehat{\Cl_K}} \prod_{i=1}^n \eps_{\chi_i ,i}  c_{f_i}\frac{|D_K|^{1/4}}{|\Cl_K|} |L(\pi_i\otimes \chi_i \Omega_{i,K} , 1/2)|^{1/2} \chi([\mathfrak{a}]).
\end{align*}
Summing this identity over a set of representatives for $\Gamma_0(N)\backslash \mathcal{Q}_{D_K}(N,r)\cong \Cl_K $, applying Theorem \ref{thm:maincomp} and using orthogonality of class group characters (i.e. the Fourier theoretic equality (\ref{eq:widemoment})$=$(\ref{eq:widemoment2})), we arrive at the wanted. 
\end{proof}
\begin{remark}
The fact that we have $||\prod_i f_i||_2$ in the error term and not, say, $L^\infty$-norms turns out to be crucial for applications to non-vanishing (see Subsection \ref{subsec:holo}).
\end{remark}

\subsection{The diagonal case}
In this subsection we will use Corollary \ref{cor:main} to calculate another family of moments. For this consider the following {\lq\lq}non-trivial diagonal{\rq\rq}:
$$ \mathbf{Wide}_{\mathrm{ntd}}(\widehat{G},2n):=\{(\chi_1,\psi_1, \ldots, \chi_n, \psi_n)\in (\widehat{G})^{2n}:  \chi_i\neq \psi_i, \prod_{i=1}^n \chi_i= \prod_{i=1}^n\psi_i\}.$$
The starting point is the following lemma. 
\begin{lemma}\label{lem:diagonalcase}
Let $G$ be a finite abelian group and $L_1,\ldots, L_n: G\rightarrow \C$ maps. Then we have
\begin{align*} 
 &\sum_{(\chi_i,\psi_i) \in\mathbf{Wide}_{\mathrm{ntd}}(\widehat{G},2n)} \prod_{i=1}^n\widehat{L_i}(\chi_i)\overline{\widehat{L_i}(\psi_i)}\\
 &=\frac{1}{|G|}\sum_{M\subset \{1,\ldots, n\}}(-1)^{|M|}  (\sum_{g\in G} \prod_{i\notin M } |L_i(g)|^2)\prod_{i\in M } (\sum_{g\in G}  |L_i(g)|^2).
\end{align*}%(\chi_1,\psi_1,\ldots, \chi_n,\psi_n)
Here $\widehat{L}: \widehat{G}\rightarrow \C$ denotes the Fourier transform given by $\chi\mapsto \frac{1}{|G|}\sum_{g\in G} L(g)\overline{\chi}(g) $. 
\end{lemma}
\begin{proof}
By the principle of inclusion and exclusion, we have 
\begin{align}
&\sum_{(\chi_i,\psi_i) \in \mathbf{Wide}_{\mathrm{ntd}}(\widehat{G},2n)} \prod_{i=1}^n\widehat{L_i}(\chi_i)\overline{\widehat{L_i}(\psi_i)} \\
=&\sum_{M\subset \{1,\ldots, n\}}(-1)^{|M|} \sum_{\substack{(\chi_1,\overline{\psi_1},\ldots, \chi_n,\overline{\psi_n}) \in\mathbf{Wide}(\widehat{G},2n)\\ \chi_i= \psi_i, i\in M}} \prod_{i=1}^n\widehat{L_i}(\chi_i)\overline{\widehat{L_i}(\psi_i)},
\end{align}
where the sum is over all subsets $M$ of $\{1,\ldots, n\}$. Furthermore, we have 
\begin{align}
&\sum_{\substack{(\chi_1,\overline{\psi_1},\ldots, \chi_n,\overline{\psi_n}) \in\mathbf{Wide}(\widehat{G},2n)\\ \chi_i= \psi_i,  i\in M}} \prod_{i=1}^n\widehat{L_i}(\chi_i)\overline{\widehat{L_i}(\psi_i)}\\
&=\left(\sum_{(\chi_i,\overline{\psi_i})_{i\notin M} \in \mathbf{Wide}(\widehat{G}, 2(n-|M|))} \prod_{i\notin M}\widehat{L_i}(\chi_i)\overline{\widehat{L_i}(\psi_i)}\right)\prod_{i\in M} \left( \sum_{\chi\in \widehat{G}}|\widehat{L_i}(\chi)|^2 \right), 
\end{align}
from which the result follows using the Fourier theoretic equality (\ref{eq:widemoment})$=$(\ref{eq:widemoment2}).
\end{proof}

From this we get the following corollary. 
\begin{cor}\label{cor:diagonalmoment}
Let $\pi_i, K, k_i$ be as in Corollary \ref{cor:main}. For $i=1,\ldots, n$, let $ \Omega_{i,K}$ be a Hecke character of $K$ of $\infty$-type $\alpha\mapsto (\alpha/|\alpha|)^{k_i}$ and $f_i\in \mathcal{B}_{k_i}(N)$ in the representation space of $\pi_i$. Then we have:

\begin{align*}
&\sum_{(\chi_i,\psi_i)\in \mathbf{Wide}_{\mathrm{ntd}}(K,2n)} \prod_{i=1}^n \eps_{\chi_i,\psi_i,f_i} |c_{f_i}|^2 L(\pi_i\otimes \chi_i \Omega_{i,K},1/2)^{1/2}L(\pi_i\otimes \psi_i \Omega_{i,K},1/2)^{1/2}\\
& = \frac{|\Cl_K|^{2n}}{|D_K|^{n/2}}\Biggr(\sum_{M\subset \{1,\ldots, n\}}(-1)^{|M|}|| \prod_{i\notin M }f_i  ||^2 \cdot  \prod_{i\in M }|| f_i ||^2\\
&\qquad \qquad \qquad+O_\eps \left( \prod_{i=1}^n || f_i||_\infty |D_K|^{-1/16}T^5n^{5}2^n (T|D_K|n)^\eps\right)\Biggr),
\end{align*}  
%&= \langle \prod_{i=1}^n f_i,1\rangle +O_\eps \left( || \prod_{i=1}^n f_i||_2 |D_K|^{-\RP}(T^{1/2}n)^{5/4-\RP} (T|D_K|n)^\eps + (T|D_K|n)^{-100}\right),
as $|D_K|\rightarrow \infty$, where $c_{f_i}=(8N_i)^{-1}c_\infty(\pi_{i,\infty},k_i)$ with $c_\infty(\pi_{i,\infty},k_i)$ as in (\ref{eq:cinfty}) and $\eps_{\chi,\psi,f_i}$ complex numbers of norm $1$.  
%as $|D_K|\rightarrow \infty$, where $f_i=\otimes f_{i,p}\in \pi_i$ are test vectors new at all finite places, of weight $k_i$ at $\infty$ and $L^2$-normalized with respect to the Petersson norm. Here we have $c_{f_i}= (8N)^{-1}c_\infty(\pi_{i,\infty},k_i)$ with $c_\infty$ as in (\ref{eq:cinfty}) and $\eps_{\chi,\psi,i}$ are all of norm $1$. 
\end{cor}
\begin{proof}
The result follows from Lemma \ref{lem:diagonalcase} combined with Corollary \ref{cor:main} by bounding the norms in the error-terms by the $L^\infty$-norms of the $f_i$.
\end{proof}
%\begin{align}
%&\sum_{\substack{\chi_1,\psi_1,\ldots, \chi_n,\psi_n \in \widehat{\Cl_{-D}}\\ \prod \chi_i=\prod \psi_i\\ \chi_i\neq \psi_i}} \prod_{i=1}^n \eps_i  L(f_i\otimes \chi_i \lambda_{i,-D},1/2)^{1/2} L(f_i\otimes \psi_i \lambda_{i,-D},1/2)^{1/2}\\
%&\rightarrow \sum_{M\subset \{1,\ldots, n\}}(-1)^{|M|}|| \prod_{i\notin M }f_i  ||^2 \cdot  \prod_{i\in M }|| f_i ||^2,% shoudl we L^2 normalize?
%\end{align}  
%as $D\rightarrow \infty$ with $D$ a fundamental discriminant such that all primes dividing $N$ splits in $\Q[\sqrt{-D}]$.

\section{Applications to non-vanishing}\label{sec:applications}
Clearly, Corollary \ref{cor:main} gives a way to produce weak simultaneous non-vanishing results (in the sense of Section \ref{sec:weaknonv}) given that one has 
\begin{equation}\label{eq:prodnot0}\left\langle \prod_{i=1}^n f_i,1\right\rangle\neq 0.\end{equation}
In this section we show non-vanishing as in (\ref{eq:prodnot0}) in a number of different cases.

The simplest case is $n=2$ and $f_1=\overline{f_2}$ (which is the one considered by Michel and Venkatesh \cite{MichVenk06}) where the period is the $L^2$-norm and thus automatically non-zero. Using our quantitive moment calculation Corollary \ref{cor:main}, we obtain a uniform version of \cite[Theorem 1]{MichVenk06} in the general weight case.  

The case $n=3$ is also very appealing since the corresponding triple periods are connected to triple convolution $L$-functions via the Ichino--Watson formula \cite{Ichino08},  \cite{Watson02}. There are some prior work obtaining non-vanishing of triple periods, which immediately give weak simultaneous non-vanishing using Corollary \ref{cor:main}. Reznikov \cite{Reznikov01} showed using representation theory that for any Maa{\ss} form $f$ of level $N$, there are infinitely many Maa{\ss} forms $f_1$ of level dividing $N$ such that $\langle f^2,f_1\rangle\neq 0$ (in the level $1$ case this was reproved by X. Li \cite{Li09} using more analytic methods). Similarly, the quantum variance computation of Luo and Sarnak \cite{LuoSarnak04} implies the following; for any Hecke--Maa{\ss} eigenform $f$ with $L(f,1/2)\neq 0$ there are $\gg K$ many holomorphic newforms $g\in\mathcal{S}_k(1)$ with $K\leq k\leq 2K$ such that $\langle y^k|g|^2,f\rangle\neq 0$ (see also \cite{SugiTsuzuki18}). One gets similar non-vanishing with $f$ a Hecke--Maa{\ss} newform using the corresponding quantum variance computation by Zhao and Sarnak \cite{SarnakZhao19}. Note that the non-vanishing results for triple periods $\langle f_1f_2f_3,1\rangle$ obtained in the above mentioned papers all have two of the forms equal. In terms of applications to non-vanishing these result are not that interesting. Motivated by this, we introduce below a method for obtaining non-vanishing for $n=3$ where all of the forms $f_1,f_2,f_3$ are different.

Finally in the holomorphic case, we can show non-vanishing of periods for general $n$ using a very soft argument.% In this section we will show examples of this an obtain what we call weak simultaneous non-vanishing. 
\subsection{The second moment case}
In this subsection, we consider the simplest case of $n=2$ in which the non-vanishing of the main term in (\ref{eq:maincomp}) is automatic. In particular this gives an improved version of \cite[Theorem 1]{MichVenk06} with uniformity in the spectral aspect and generalizes the results to general weights.

\begin{cor}\label{cor:n=2}
Let $N$ be a fixed square-free integer and $\eps>0$. Let $\pi$ be a cuspidal automorphic representation of $\GL_2(\A)$ of level $N$, spectral parameter $t_\pi$ and even lowest weight $k_\pi$. Let $k$ be an even integer such that $|k|\geq k_\pi$ and put $T=|t_\pi|+k+1$. 

Then there exists a constant $c=c(N, \eps )>0$ such that for any imaginary quadratic field $K$ such that all primes dividing $N$ splits in $K$ with discriminant $|D_K|\geq c T^{160/3+\eps}$ (resp. $|D_K|\geq c T^{22+\eps}$ if $N=1$) , we have
$$ \#\{\chi\in \widehat{\Cl_K}: L(\pi\otimes \chi \Omega_{K},1/2)\neq 0\}\gg_\pi \begin{cases} |D_K|^{1/1058},& \text{if $\pi$ is d.s,}\\ |D_K|^{1/2648},& \text{if $\pi$ is p.s,}\end{cases}$$
where $\Omega_K$ is a Hecke character of $K$ of conductor $1$ and $\infty$-type $\alpha\mapsto (\alpha/|\alpha|)^k$.
\end{cor}
\begin{proof}
Let $\pi$ be as in the corollary above. We apply Corollary \ref{cor:main} with the error-term coming from Remark \ref{remark:ET} and with $\pi_1=\pi_2=\pi$ and $f_1=\overline{f_2}$ belonging to $\pi$ of weight $k\geq k_\pi$. In this special case it is clear that we can truncate the spectral expansion (\ref{eq:experror}) at $t_u\ll T^{1+\eps} |D_K|^\eps$ at a negligible error since the lower bound $||f_1f_2||_1=||f_1||_2=1$ (for any $f_1$ as above) is automatic. Thus both in the (raised) holomorphic and Maa{\ss} case we have the following error-terms: 
$$ O_\eps\left( |D_K|^{-1/16} T^{20/6} (|D_K|T)^\eps   \right) \quad \text{for general level $N$},$$
and   
$$ O_\eps\left(  |D_K|^{-1/12} T^{11/6} (|D_K|T)^\eps  \right) \quad \text{for level $N=1$}.$$
From this we see that for $|D_K|\geq c T^{160/3+\eps}$ (resp. $|D_K|\geq c T^{22+\eps}$), the RHS of (\ref{eq:maincomp}) is non-zero. Thus the LHS (\ref{eq:maincomp}) is also non-zero and satisfies $\gg_{\eps,k} |D_K|^{1/4-\eps}$ using the lower bound $|\Cl_K|\gg_\eps |D_K|^{1/2-\eps}$ due to Siegel. Now the result follows directly using the subconvexity bounds for Rankin--Selberg $L$-functions due to Michel \cite{Michel04} and Harcos--Michel \cite{HarcosMichel06}.  
\end{proof}

\subsection{Triple products of Maa{\ss} forms}
A very attractive case of Corollary \ref{cor:main} is $n=3$ where the non-vanishing of $\langle f_1f_2f_3,1\rangle$ is equivalent to the non-vanishing of the triple convolution $L$-function $L(\pi_1\otimes \pi_2\otimes \pi_3,1/2)$ due to the Ichino--Watson formula \cite{Ichino08}, \cite{Watson02}. In this section we introduce a soft method (relying on results of Lindenstrauss and Jutila--Motohashi) to derive non-vanishing results in the case where  $f_1,f_2,f_3$ are all Maa{\ss} forms of level $1$. 
%The idea is to spectrally expand $|f_1|^2|f_2|^2$ and the main difficulty is to control the contribution from the continuous spectrum.

By the spectral expansion for $L^2(\SL_2(\Z)\backslash \Hb)$ \cite[Theorem 7.3]{Iw} we have
\begin{align}\label{eq:spectralexp}
|| f_1f_2 ||_2^2=\langle f_1f_2, f_1f_2\rangle=\sum_{f\in \mathcal{B}_{0}(1)} |\langle f_1f_2, f \rangle|^2 +\frac{1}{4\pi}\int_{\R}  |\langle f_1f_2, E_t \rangle|^2 dt,
\end{align}
where $E_t(z)=E(z,1/2+it)$ is the non-holomorphic Eisenstein series of level $1$. Using the Ichino--Watson formula \cite{Ichino08}, \cite{Watson02} (which in the Eisenstein case reduces to Rankin--Selberg), we have
\begin{equation*}
|\langle f_1f_2, f \rangle|^2=\frac{L(f_1\otimes f_2 \otimes f,1/2)}{8L(\sym^2 f_1,1)L(\sym^2 f_2,1)L(\sym^2 f,1)} h(t_{f_1},t_{f_2},t_{f})\end{equation*}
and 
\begin{equation*}
|\langle f_1f_2, E_t \rangle|^2=\frac{|L(f_1\otimes f_2 ,1/2+it)|^2}{4L(\sym^2 f_1,1)L(\sym^2 f_2,1)|\zeta(1+2it)|^2} h(t_{f_1},t_{f_2},t)\end{equation*}
where
\begin{equation*}%\pi^{3/2}
h(t_1,t_2,t_3)=\frac{\prod_{\pm} \Gamma\left(\frac{1}{4} \pm \frac{it_{1}}{2}\pm  \frac{it_{2}}{2}\pm  \frac{it_3}{2}\right)}{ |\Gamma\left(\frac{1}{2}+it_{1}\right)|^2|\Gamma\left(\frac{1}{2}+it_{2}\right)|^2|\Gamma\left(\frac{1}{2}+it_{3}\right)|^2}.
\end{equation*}
Here the product is over all 8 combinations of signs. If we fix $t_1$, then it is standard using Stirling's approximation to prove that for $t_2,t_3\gg 1$, we have
$$ h(t_1,t_2,t_3)\ll_{t_1} e^{-\pi |t_2-t_3|}(1+|t_2-t_3|)^{-1}(1+t_2+t_3)^{-1}.  $$
This shows that the contribution from respectively, $|t-t_{f_2}|\geq (t_{f_2})^\eps$ and $|t_f-t_{f_2}|\geq (t_{f_2})^\eps$ in (\ref{eq:spectralexp}) is negligible. 

We would like to show that actually all of the contribution from the Eisenstein part in (\ref{eq:spectralexp}) is negligible. This is connected to the subconvexity problem for Rankin--Selberg $L$-functions in a conductor dropping region and is thus very difficult. We can however get unconditional results if we keep $f_1$ fixed and average over $f_2$ using the following result due to Jutila and Motohashi  \cite[(3.50)]{JutilaMoto05}.
\begin{thm}[Jutila--Motohashi]\label{thm:JutMoto} Let $f_1\in \mathcal{B}_{0}(1)$ be fixed. Then we have
\begin{equation}  \sum_{ |t_{f_2}-T|\leq T^\eps} |L(f_1\otimes f_2,1/2+it)|^2\ll_{\eps} T^{1+\eps},    \end{equation}
uniformly for $| t-T|\ll T^{\eps}$.
\end{thm}
Strictly speaking \cite{JutilaMoto05} only deals with the case where $f_1$ is an Eisenstein series, but (as remarked in \cite[p. 3]{BlomerHolo10}) the same estimate follows in the case of Maa{\ss} forms using the exact same argument relying on the spectral large sieve.% (CHECK THIS?). 

From Theorem \ref{thm:JutMoto} it follows that for any $\delta>0$, we have that 
\begin{equation}\label{eq:MotoJutila} \int_{|t-t_{f_2}|\leq (t_{f_2})^\eps} |L(f_1\otimes f_2,1/2+it)|^2 dt \leq T^{1-\delta} \end{equation}
for all but at most $O_\eps (T^{\delta+\eps})$ Maa{\ss} forms $f_2$ with $|t_{f_2}-T|\leq T^\eps$. 

Recalling the estimates $t_f^{-\eps}\ll_\eps L(\sym^2 f,1)\ll_\eps t_f^\eps$, we conclude combining all of the above that for any $f_2$ satisfying (\ref{eq:MotoJutila}), we have
\begin{equation}\label{eq:spectralwerror}
|| f_1f_2 ||_2^2=\sum_{|t_f-T|\leq T^\eps} |\langle f_1f_2, f \rangle|^2+O_\eps(T^{-\delta+\eps}).
\end{equation} 
By QUE for Maa{\ss} forms due to Lindenstrauss \cite{Lindenstrauss06} (with key input by Soundararajan \cite{SoundAnn10}), we know that 
$$|| f_1f_2 ||_2\rightarrow ||f_1||_2\neq 0,\quad \text{and}\quad \langle f_1f_2, f_2 \rangle\rightarrow \langle f_1, \frac{3}{\pi}\rangle =0,$$ as $t_{f_2}\rightarrow \infty$. Thus we conclude from  (\ref{eq:spectralwerror}) that for $T$ large enough there is some $f_3\neq f_2$ with $|t_{f_3}-T|\leq T^\eps$ such that $\langle f_1f_2, f_3 \rangle\neq 0$. Furthermore, we obtain a lower bound for free using Weyl's law; 
$$\#\{ f\in \mathcal{B}_0(1): |t_f-T|\leq T^\eps \}\asymp T^{1+\eps}.$$ 
From this we obtain the following result.

\begin{prop}\label{prop:lowerboundperiod}
Let $f_1\in \mathcal{B}_{0}(1)$ be fixed and $\eps>0$. Then for $T>0$ large enough (depending on $f_1$ and $\eps$), we have that for all but $O_\eps(T^{2\eps})$ of $f_2\in \mathcal{B}_{0}(1)$ satisfying  $|t_{f_2}-T|\leq T^\eps$, there exists some $f_3 \in \mathcal{B}_{0}(1)$ not equal to $f_2$ with $|t_{f_3}-T|\leq T^\eps$ such that 
$$ |\langle f_1f_2,f_3\rangle|\gg_\eps ||f_1 f_2||_2/ T^{1/2+\eps}. $$
\end{prop}
From this we deduce the non-vanishing result in Corollary \ref{cor:nonvanishinglevel1}.
\begin{proof}[Proof of Corollary \ref{cor:nonvanishinglevel1}]
Let $f_2,f_3$ be as in Proposition \ref{prop:lowerboundperiod}. Then we apply Corollary \ref{cor:main} (in the level $1$ case) with $n=3$, $k_1=k_2=k_3=0$ and test vectors $f_1,f_2,f_3$. We observe that we  have
$$ ||f_1f_2f_3||_2\, |D_K|^{-1/12} T^2 (|D_K|T)^\eps \ll ||f_1f_2||_2\, |t_{f_3}|^{5/12+\eps} |D_K|^{-1/12} T^2 (|D_K|T)^\eps,    $$
by the sup-norm bound due to Iwaniec and Sarnak \cite{IwSa}.  Thus we see that if $|D_K|\gg_{f_1,\eps} T^{35+\eps}$, the error-term in the asymptotic (\ref{eq:maincomp}) (with exponents as in (\ref{eq:errortermscormain})) is strictly less than $\langle f_1f_2f_3,\frac{3}{\pi}\rangle$. Thus we conclude that the LHS of (\ref{eq:maincomp}) is non-vanishing and satisfies $\gg_{\eps,T} |D_K|^{3/4-\eps}$ (using the lower bound $|\Cl_K|\gg_\eps |D_K|^{1/2-\eps}$ again). Now by the subconvexity estimate for $L(f_i\otimes \theta_{\chi_i},1/2)$ due to Harcos and Michel \cite[Theorem 1]{HarcosMichel06} (where $\theta_{\chi_i}$ is the holomorphic theta series associated to the Hecke character $\chi_i$), we get the wanted quantitive non-vanishing result as $|D_K|\rightarrow \infty$. 
\end{proof}

\subsection{The holomorphic case}\label{subsec:holo}
Consider Corollary \ref{cor:main} in the case where $\pi_1,\ldots, \pi_n$ are all holomorphic discrete series representations of $\GL_2$ and $k_i=k_{\pi_i}>0$. Furthermore, pick $f_i=y^{k_i/2} g_i$ with $g_i \in \mathcal{S}_{k_i}(N)$ a holomorphic Hecke newform. Then we know that 
$$\prod_{i=1}^ng_i\in  \mathcal{S}_{k}(N),$$
where $k=\sum_i k_i$ (which might not be a Hecke--Maa{\ss} eigenform(!)). A basis $\mathcal{B}_{k,\mathrm{hol}}(N)$ for $ \mathcal{S}_{k_i}(N)$ is given by $\nu^*_{d,N'} y^{k/2} g$, where $g\in  \mathcal{S}_{k}(N')$ is a Hecke newform and $dN'| N$. This implies that 
$$ ||y^k  \prod_{i=1}^n g_i||_2^2= \sum_{\substack{u_1,u_2\in \mathcal{B}_{k,\mathrm{hol}}(N)}} \langle u_1,u_2\rangle \langle y^{k/2} \prod_{i=1}^n g_i, u_1\rangle\overline{ \langle y^{k/2}\prod_{i=1}^n g_i, u_2\rangle} $$
Since any two $u_1,u_2\in \mathcal{B}_{k,\mathrm{hol}}(N)$ are orthogonal (wrt. to the Petersson innerproduct) if the underlying Hecke newforms are different and the dimension of $\mathcal{S}_{k}(N')$ is $\ll_N k$, we conclude the following.
\begin{prop}\label{prop:non0periodholo} Let $N$ be a fixed positive integer and let $k_1,\ldots, k_n\in 2\Z_{>0}$ be even integers. For $i=1,\ldots, n$, let $g_i \in \mathcal{S}_{k_i}(N)$ be a holomorphic Hecke newform of level $N$ and weight $k_i$. Then there exists some $\nu^*_{d,N'} y^{k/2}g\in \mathcal{B}_{k,\mathrm{hol}}(N)$ with $k=k_1+\ldots +k_n$ such that 
$$  \langle  \prod_{i=1}^n y^{k_i/2}g_i, \nu^*_{d,N'} y^{k/2} g\rangle\gg  || \prod_{i=1}^n y^{k_i/2}g_i||_2/k^{1/2}. $$
\end{prop}
Combining this with Corollary \ref{cor:main} we obtain the following non-vanishing result.
\begin{cor}\label{cor:holononvgeneral}
%(such that $\mathcal{S}_{\kappa}(\Gamma_0(N))$ only contain newforms(?))
Let $N$ be a fixed square-free integer and let $k_1,\ldots, k_n\in 2\Z_{>0}$ be even integers. For $i=1,\ldots,n$, let $\pi_i$ be automorphic representations corresponding to holomorphic newforms $g_i\in \mathcal{S}_{k_i}(N)$ and put $k=\sum k_i$. Then there exists a constant $c=c( N,\eps)>0$ such that for any imaginary quadratic field $K$ such that all primes dividing $N$ splits in $K$ and the discriminant satisfies $|D_K|\geq c (\max_i k_i)^{40}n^{80}k^{12+\eps}$, we have 
\begin{align*} & \# \{ (\chi_1,\ldots, \chi_n)\in \mathbf{Wide}(K,n+1), g\in \mathcal{B}_{k,\mathrm{hol}}(\Gamma_0(N)) : \\
& L(\pi_1\otimes \chi_1 \Omega_{i,K },1/2)\cdots L(\pi_{n}\otimes \chi_{n} \Omega_{n,K},1/2) L(\pi_{g}\otimes  \chi_{n+1}\Omega_{n+1,K},1/2) \neq 0 \}\\
&\hspace*{+9cm}  \gg_k |D_K|^{(n+1)/2115}, 
\end{align*} 
where $k=\sum_i k_i$ and $\Omega_{i,K}$ are Hecke characters of $K$ with $\infty$-types $x\mapsto (x/|x|)^{k_i}$, and $ \Omega_{n+1,K}=\prod_{i=1}^{n} \Omega_{i,K}$.
\end{cor}
\begin{proof} For $i=1,\ldots, n$, let $f_i=y^{k_i/2} g_i$, and let $f=\nu^*_{d,N'} y^{k/2}g\in  \mathcal{B}_{k,\mathrm{hol}}(\Gamma_0(N))$ be as in Proposition \ref{prop:non0periodholo}. Then we have the following sup-norm bound due to Xia \cite{Xia07} (or more precisely the natural extension to general level);
$$||f ||_\infty \ll_\eps k^{1/4+\eps}.$$ 
Thus we conclude that 
$$  ||f \prod_{i=1}^n f_i||_2\ll_\eps k^{1/4+\eps } ||\prod_{i=1}^n f_i||_2, $$
 and thus we see that there is some constant depending only on $N$ such that as soon as 
\begin{align*}
|D_K|^{1/16} \gg_{N, \eps }  \left(\max_{i=1,\ldots, n} k_i\right)^{5/2}n^{5}k^{1/4+1/2+\eps}, 
\end{align*}
then the RHS of (\ref{eq:maincomp}) is non-zero. Thus the LHS (\ref{eq:maincomp}) is also non-zero and is $\gg_{\eps,k} |D_K|^{n/4-\eps}$ using the lower bound $|\Cl_K|\gg_\eps |D_K|^{1/2-\eps}$. 

Finally, since all of the $f_i$ are holomorphic we can employ the subconvexity bound for Rankin--Selberg $L$-functions $L(f_i\otimes \theta_{\chi_i\Omega_{i,K}},1/2)$ due to Michel \cite{Michel04}, where $\theta_{\chi_i\Omega_{i,K}}$ is the holomorphic theta series associated to the Hecke character $\chi_i\Omega_{i,K}$ defined in Subsection \ref{sec:Heckechar}.  Finally we use that 
$$L(f\otimes \theta_{\chi\Omega_{n+1,K}},1/2)= \overline{L(f\otimes \theta_{\overline{\chi}\overline{\Omega}_{n+1,K}},1/2)},$$
to get rid of the conjugate in the last Rankin--Selberg $L$-functions. This gives the wanted qualitative lower bound for the non-vanishing.%(and \cite{DFI94} in the case of genus character) Not neccesary since all of the theta series are cuspidal
\end{proof}
%In the case of level $1$, we can do better since in this case all forms are new and thus the choice of basis above is orthogonal (wrt. Petersson inner product). Thus in this case we can find a Hecke--Maa{\ss}form $f\in \mathcal{S}_k(\Gamma_0(1))$ such that 
%$$  \langle \prod_{i=1}^n f_i, y^{k/2} f\rangle\gg  || \prod_{i=1}^n f_i||_2/k^{1/2}. $$
In the special case of level $1$, we can do slightly better.
\begin{proof}[Proof of Corollary \ref{cor:nonvanishinglevel1holo}]
Using the improved error-term in Corollary \ref{cor:main} in the case of level $1$ holomorphic forms, we see that the RHS of (\ref{eq:maincomp}) is non-zero as soon as 
$$ |D_K|^{1/12}\gg_{N, \eps }  (\max_{i=1,\ldots, n} k_i)n^{2}k^{3/4+\eps}. $$
Using the trivial estimates $n \leq k$ and $\max_i k_i \leq k$, we conclude Corollary \ref{cor:nonvanishinglevel1holo}.% using the trivial estimate $n\leq k_1+\ldots +k_n$. 
\end{proof}
%\begin{remark}
%Dependence on $k$ is polynomial??
%\end{remark}

\bibliography{/Users/asbjornnordentoft/Documents/MatematikSamlet/Bib-tex/mybib}
\bibliographystyle{alpha}

\end{document}